\documentclass[final,leqno]{siamltex}
\title{On the relaxed greedy deterministic row and column iterative methods}
\author{Nian-Ci Wu\footnotemark[1]~,~\
Ling-Xia Cui\footnotemark[2]
\and
Qian Zuo\footnotemark[2]\  \footnotemark[3]}

\begin{document}

\maketitle
\renewcommand{\thefootnote}{\fnsymbol{footnote}}
\footnotetext[1]{School of Mathematics and Statistics, South-Central University for Nationalities, Wuhan 430074, China.}
\footnotetext[2]{School of Mathematics and Statistics, Wuhan University, Wuhan 430072, China.}
\footnotetext[3]{Corresponding author. E-mail addresses: {\sf nianciwu@scuec.edu.cn}, {\sf lxcui@whu.edu.cn} and {\sf zuoqian@whu.edu.cn}.}
\renewcommand{\thefootnote}{\arabic{footnote}}

\begin{abstract}
For solving the large-scale linear system by iteration methods, we utilize the Petrov-Galerkin conditions and relaxed greedy index selection technique, and provide two relaxed greedy deterministic row (RGDR) and column (RGDC) iterative methods, in which one special case of RGDR reduces to the fast deterministic block Kaczmarz method proposed in Chen and Huang (Numer. Algor., 89: 1007-1029, 2021). Our convergence analyses reveal that the resulting algorithms all have the linear convergence rates, which are bounded by the explicit expressions. Numerical examples  show that the proposed algorithms are more effective than the relaxed greedy randomized row and column iterative methods.
\end{abstract}


\pagestyle{myheadings}
\thispagestyle{plain}
\markboth{On the relaxed greedy deterministic row and column iterative methods}{Nian-Ci Wu, Ling-Xia Cui and Qian Zuo}

\section{Introduction}\label{sec1}\noindent
Let $A$ be a real $m$-by-$n$ ($m\geq n$) matrix and $\bb$ be a real $m$-dimensional right-hand side. We consider the iterative solution for the system of linear equations
\begin{equation}\label{eq:Ax=b}
  A\bx= \bb, ~{\rm with}~A\in\Rc^{m\times n}~{\rm and}~\bb\in\Rc^{m},
\end{equation}
where $\bx$ is the $n$-dimensional unknown vector.

Projection algorithm \cite{Saad2000} is a kind of classic while effective iterative solver for computing an approximate solution for \eqref{eq:Ax=b}. As we know, many existing practical projection iterative algorithms are under the framework of Petrov-Galerkin conditions, such as the Kaczmarz and coordinate descent (CD) methods, see \cite{15GR,Saad2000,20WX3,20WX4,XZ17} and the references therein. Let the constrained subspace and  the search subspace correspond to $\mathcal{L}={\rm span}\{ Y\}$ and  $\mathcal{K}= {\rm span}\{ Z\}$, respectively, where $Y$ and $Z$ are two parameter matrices. It leads  to the following iteration scheme \cite{15GR,XZ17}.
\begin{equation}\label{eq:PG-iterate}
\bx^{(k+1)}  = \bx^{(k)} + Z(Y^T A Z)^\dag Y^T (\bb-A\bx^{(k)})
\end{equation}
for $k=0,1,2,\cdots$. By varying the parameter matrices $Y$ and $Z$,
the above two well-known single row and column methods are recovered as follows.
\begin{enumerate}[$\diamond$]
\setlength{\itemindent}{0.5cm}
\addtolength{\itemsep}{-0.1em} 
\item Kaczmarz method ($Y=\bmu_i$, $Z=A^TY$, and $i=1,2,\cdots,m$):
\begin{equation}\label{eq:Kacz-iterate}
  \bx^{(k+1)}  = \bx^{(k)} + \frac{\bmu_i^T(\bb - A\bx^{(k)})}{\BT{\bal_i}}\bal_i,
\end{equation}
where $\bmu_i^T$ and $\bal_i^T$  are the $i$-th rows of the  identity matrix with size $m$ and $A$, respectively;\vskip 1.25ex
\item CD method ($ Z=\bnu_j$, $Y=AZ$, and $j=1,2,\cdots,n$):
\begin{equation}\label{eq:CD-iterate}
\bx^{(k+1)}  = \bx^{(k)} + \frac{\bbe_j^T(\bb - A\bx^{(k)})}{\BT{\bbe_j}}\bnu_j,
\end{equation}
where $\bnu_j $ and $\bbe_j$  are the $j$-th columns of the  identity matrix with size $n$ and $A$, respectively.
\end{enumerate}

Randomization has several benefits, for example, the resulting algorithm is easy to analyze, simple to implement, and often effective in practice \cite{19BW, 19GMMN, 10LL, 20LZ, 20ZG}. The randomized methods were proposed \cite{93HM}, but with no explicit proofs of convergence until the celebrated paper \cite{SV09},  where Strohmer and Vershynin  proved that the randomized Kaczmarz (RK) method converges linearly in expectation, with a rate directly related to geometric properties of the matrix $A$. The randomized CD (RCD) method was given by Leventhal and  Lewis \cite{10LL}.

 One  key ingredient of guaranteeing fast convergence for RK and RCD is to construct an appropriate probability criterion used to select  $\bal_i^T$ and $\bbe_j$ from $A$, respectively. The promising adaptive index selection was introduced in  Bai and Wu's series of work \cite{18BW1,18BW2,19BW},  where the authors  applied it to RK and RCD, and presented the greedy RK (GRK \cite{18BW1}, RGRK \cite{18BW2}) and RCD (GRCD \cite{19BW}) methods.
The relaxed version of GRCD (RGRCD) is given by \cite{20ZG}.  The relaxed greedy row and column selection strategies were enlarged to more general cases by Gower et al. \cite{19GMMN}.

The adaptive index selection immediately  results in many papers devoted to various aspects of the new algorithm \cite{20DSS, 20LZ, 19Nec, 20NZ}.
Li and Zhang gave a greedy block Gauss-Seidel method, where all the coordinates satisfy a greedy rule \cite{20LZ}.  Based on the block Kaczmarz iteration scheme \cite{15GR, XZ17} and a modified greedy row selection from \cite{18BW1, 18BW2}, the greedy block Kaczmarz (GBK) was proposed by Niu and Zheng \cite{20NZ}. This method needs to calculate the Moore-Penrose pseudoinverse of the row submatrix. Several pseudoinverse-free methods can be found in \cite{20DSS,20LZ,19Nec}. Recently, a fast deterministic block Kaczmarz (FDBK) method, according to the greedy criterion of row selections, was presented by Chen and Huang \cite{21CH}.

In this work, we utilize the Petrov-Galerkin conditions \cite{Saad2000} and relaxed greedy index selection technique \cite{18BW2,20ZG}, and present several new relaxed greedy deterministic row and column methods. The paper is organized as follows. In the rest of this section, we give some notation. In Section \ref{sec:RGRR+RGRC} we briefly recap the classical relaxed greedy randomized row and column methods for solving the linear system.  In Section \ref{sec:RGDR+RGDC} we present two relaxed greedy deterministic row and column methods together with their convergence theories.  In Section \ref{sec:ER} some numerical examples are provided to demonstrate the theoretical results.

{\bf Notation.} For an integer $n\geq 1$, let $ [n]:=\{1,\cdots,n\}$.  The identity matrix of size $n$ is given by $I_n$. Throughout the paper, all vectors are assumed to be column vectors.
For any matrix $M$,  we use $M^T$, $\Rc(M)$, $\Nc(M)$, $\sigma_{\max}(M)$, $\sigma_{\min}(M)$, $M_{i, :}$, $M_{:,j}$ and $M_{ij}$ to denote the transpose, the range, the null space, the largest,  the smallest nonzero singular values, the $i$-th row, $j$-th column and the $(i,j)$-th entry,  respectively.
We remark that $M_{I_k, :}$ and $M_{:, J_k}$ stand for the row and column submatrices of $M$ indexed by index sets $I_k$ and $J_k$, respectively.
For any vector $\bx$,  we use $\bx^T$ and $x_i$ to denote the transpose and the $i$-th entry, respectively. We can uniquely write a vector $\bb\in \Rc^m$ as $\bb= \bb_{\textsc{r}} + \bb_{\textsc{n}}$,  where $\bb_{\textsc{r}}=AA^{\dag}\bb$ and $\bb_{\textsc{n}}=(I_m-AA^{\dag})\bb$ are the projections of $\bb$ onto $\Rc(A)$ and $\Nc(A^T)$, respectively. We use $|\Omega|$ to denote the cardinality of a set $\Omega$.

\section{The relaxed greedy randomized row and column methods}\label{sec:RGRR+RGRC}
In this section, we give a brief description of the relaxed greedy randomized row and column methods, such as RGRK \cite{18BW2}  and RGRCD \cite{20ZG}.

We assume that there exists an $\bx^{\ast}$ satisfying $ A\bx^{\ast}= \bb$.  By the Pythagorean theorem, the squared errors of Kaczmarz and CD severally satisfy
\begin{align*}
  \BT{\bx^{(k+1)} - \bx^{\ast}}
   &= \BT{\bx^{(k)} - \bx^{\ast}} -\frac{|b_i - \bal_i^T \bx^{(k)}|^2}{\BT{\bal_i}} \\
   &:= \BT{\bx^{(k)} - \bx^{\ast}} - \psi_i(\bx^{(k)})
\end{align*}
and
\begin{align*}
  \BT{A\bx^{(k+1)} - A\bx^{\ast}}
  & = \BT{A\bx^{(k)} - A\bx^{\ast}} -\frac{|\bbe_j^T(\bb - A\bx^{(k)})|^2}{\BT{\bbe_j}}\\
  &:= \BT{A\bx^{(k)} - A\bx^{\ast}} -\varphi_j(\bx^{(k)}),
\end{align*}
where  $\psi_i(\bx^{(k)})$ and $\varphi_j(\bx^{(k)})$ are the losses of indices $i\in [m]$ and $j\in [n]$, respectively. This implies that we may  select the row index $i$ or column index $j$ such that the corresponding loss is as large as possible.

The main idea of adaptive index selection strategies, given by Bai et al.  \cite{18BW1,18BW2,19BW} and Zhang et al. \cite{20ZG},  is to choose the index in a compromise way, whose loss is the convex combination of the largest and mean ones, which is described as follows.
\vskip 0.15cm
\noindent{\bf The relaxed greedy row and column selection} \cite{18BW1,18BW2,19BW,20ZG}.
{\it For $k=0,1,2,\cdots$, introducing two relaxation parameters $\theta_1,\theta_2\in [0,1]$, the row and column index sets are separately determined by
\begin{equation}\label{eq:RGRS}
  U_k=\left\{i_k \Big| \psi_{i_k}(\bx^{(k)}) \geq
  \theta_1\cdot \max_{i\in [m]}\left\{ \psi_i(\bx^{(k)}) \right\}
  +(1-\theta_1)\cdot \sum_{i\in [m]} \omega_i \psi_{i}(\bx^{(k)}), i_k\in [m]\right\}
\end{equation}
with $\omega_i=\bT{\bal_i}^2 / \BF{A}$ and
\begin{equation}\label{eq:RGCS}
  V_k=\left\{j_k \Big| \varphi_{j_k}(\bx^{(k)}) \geq
  \theta_2\cdot \max_{j\in [n]}\left\{ \varphi_j(\bx^{(k)}) \right\}
  +(1-\theta_2)\cdot \sum_{j\in [n]} \varpi_j \varphi_{j}(\bx^{(k)}), j_k\in [n]\right\},
\end{equation}
with $\varpi_j = \bT{\bbe_j}^2 / \BF{A}$.}
\vskip 0.15cm

\noindent Then, the RGRK and RGRCD methods can be list as follows. For more details and their implementations, we refer to \cite{18BW2,20ZG}.

\vskip 0.15cm
\noindent{\bf The RGRK method} \cite[Method 1]{18BW2}.  {\it Given an initial vector $\bx^{(0)} \in \Rc^{n}$ and a relaxation  parameter $\theta_1$, for $k=0,1,2,\cdots$, until the iteration sequence $\left\{\bx^{(k )}\right\}_{k=0}^{\infty}$ converges,
we first construct the index set $U_k$  according to formula \eqref{eq:RGRS}  and then update $\bx^{(k+1)}$ in accordance with
 \begin{equation*}
 \bx^{(k+1)}  = \bx^{(k)} + \frac{b_{i_k} - \bal_{i_k}^T \bx^{(k)}}{\BT{\bal_{i_k}}}\bal_{i_k},
\end{equation*}
where the row index  $i_k$ is from $U_k$ and selected with probability
\begin{equation*}
 \bP({\rm Row} = i_k) = \frac{ \big| \tilde{r}^{(k)}_{i_k} \big|^2 }{\BT{{ \tilde{\bm r}}^{(k)}}}
 \quad with \quad
 \tilde{r}^{(k)}_{i} =
 \left\{
 \begin{array}{cll}
   b_{i}-\bal_i^T \bx^{(k)}&, & if~i\in    U_k,\\
   0                       &, & if~i\notin U_k.
 \end{array}
 \right.
\end{equation*}
}

\vskip 0.15cm

\noindent{\bf The RGRCD method} \cite[Algorithm 2]{20ZG}.  {\it Given an initial vector $\bx^{(0)} \in \Rc^{n}$ and a relaxation  parameter $\theta_2$, for $k=0,1,2,\cdots$, until the iteration sequence $\left\{\bx^{(k )}\right\}_{k=0}^{\infty}$ converges,
we first construct the index set $V_k$  according to formula \eqref{eq:RGCS}  and then update $\bx^{(k+1)}$ in accordance with
 \begin{equation*}
 \bx^{(k+1)}  = \bx^{(k)} + \frac{\bbe_{j_k}^T(\bb - A\bx^{(k)})}{\BT{\bbe_{j_k}}}\bnu_{j_k},
\end{equation*}
where the  column index  $j_k$ is from $V_k$ and selected with probability
\begin{equation*}
 \bP({\rm Col} = j_k) = \frac{ \big| \tilde{s}^{(k)}_{j_k} \big|^2 }{\BT{{\tilde{\bm s}}^{(k)}}}
 \quad with \quad
 \tilde{s}^{(k)}_{j} =
 \left\{
 \begin{array}{cll}
   \bbe_{j}^T(\bb - A\bx^{(k)})&, & if~j\in     V_k,\\
   0                           &, & if~j\notin  V_k.
 \end{array}
 \right.
\end{equation*}
}
\vskip 0.15cm
\noindent  It was shown that the mean squared errors in RGRK and RGRCD admit the following estimates.

\vskip 1em

\begin{theorem}\label{thm:RGRK}(\cite[Theorem 2.1]{18BW2})
Let the linear system $A\bx=\bb$,  with the coefficient matrix $A \in \Rc^{m\times n}$ and the right-hand side $\bb \in \Rc^{m}$, be consistent. Then  the iteration sequence $\left\{\bx^{(k )}\right\}_{k=0}^{\infty}$, generated by the  RGRK  method starting from an initial guess  in the column space of $A^T$, converges to the unique least norm solution $\bx^{\ast}=A^{\dag}\bb$ in expectation. Moreover, for $k\geq 1$, the solution error in expectation for the iteration sequence $\left\{\bx^{(k )}\right\}_{k=0}^{\infty}$ obeys
\begin{equation} \label{eq:thm-RGRK}
\bE \BT{\bx^{(k+1)} - A^\dag \bb}\leq
\left(1 - \tau
\frac{\sigma_{\min}^2(A)}{\BF{A}}   \right)
\bE \BT{\bx^{(k)} - A^\dag \bb},
\end{equation}
 where
 $\tau  = \theta_1 \BF{A}/\epsilon   + (1-\theta_1)$ and
$\epsilon  = \BF{A} - \min_{i\in [m]}\left\{\BT{\bal_i}\right\}$.
\end{theorem}

\vskip 1em

\begin{theorem}\label{thm:RGRCD}(\cite[Theorem 1]{20ZG})
Consider the large linear least-squares problem $\min\limits_{\bx\in \Rc^{n}}\BT{\bb-A\bx}$, where $A\in\Rc^{m\times n}$ ($m\geq n$) is of full column rank and $\bb\in\Rc^{m}$ is a given vector. Then, it holds that
the iteration sequence $\left\{\bx^{(k)}\right\}_{k=0}^{\infty}$, generated by the  RGRCD method starting from any initial guess $\bx^{(0)}\in \Rc^n$, converges to the unique least-squares solution  $\bx^{\ast}=A^{\dag}\bb$ in expectation and for $k\geq 1$, the solution error in expectation for the iteration sequence $\left\{\bx^{(k )}\right\}_{k=0}^{\infty}$ satisfies
\begin{equation} \label{eq:thm-RGRCD}
\bE \BT{A\bx^{(k+1)} - AA^\dag \bb}\leq
\left(1 - \gamma
\frac{\sigma_{\min}^2(A)}{\BF{A}}   \right)
\bE \BT{A\bx^{(k)} - AA^\dag \bb},
\end{equation}
 where
 $\gamma  = \theta_2 \BF{A}/\varepsilon + (1-\theta_2)$ and
$\varepsilon = \BF{A} - \min_{j\in [n]}\left\{\BT{\bbe_j}\right\}$.
\end{theorem}
\section{The relaxed greedy deterministic row and column methods} \label{sec:RGDR+RGDC}
In this section, we present two iterative methods and call them relaxed greedy deterministic  row (RGDR) and column (RGDC)  methods.

Let $\bet\in\Rc^{m}$ and $\bxi\in\Rc^{n}$ be two non-zero vectors.  According to Petrov-Galerkin conditions, we give the following multiple rows and columns iteration formats.

\begin{enumerate}[$\diamond$]
\setlength{\itemindent}{0.5cm}
\addtolength{\itemsep}{-0.1em} 
\item Multiple rows method ($Y=\bet$ and $Z=A^TY$):
\begin{equation}\label{eq:GK}
  \bx^{(k+1)}  = \bx^{(k)} + \frac{\bet^T(\bb - A\bx^{(k)})}{\BT{A^T \bet}} A^T \bet;
\end{equation}
\vskip 1.25ex
\item Multiple columns method ($ Z=\bxi$ and $Y=AZ$):
\begin{equation}\label{eq:GLS}
\bx^{(k+1)}  = \bx^{(k)} + \frac{\bxi^T A^T (\bb - A\bx^{(k)})}{\BT{A\bxi}}\bxi.
\end{equation}
\end{enumerate}
Note that $\bet^T A$ can be viewed as a linear combination of all rows of $A$.  In particular, if $\bet$ is a Gaussian vector, i.e.,  $\bet=[\eta_i]\in \Rc^{m}$ with $\eta_i\sim \Nc(0,1)$, Gower et al. \cite{15GR} called  \eqref{eq:GK} the Gaussian Kaczmarz method. Similarly,  $A\bxi$ can be viewed as a linear combination of all columns of $A$. \eqref{eq:GLS} is denoted by the Gaussian least-squares method in  \cite{15GR} if $\bxi$ is a Gaussian vector, i.e.,  $\bxi=[\xi_j]\in \Rc^{n}$ with $\xi_j\sim \Nc(0,1)$. For more row and column iteration formats, we refer to \cite{XZ17}.

\subsection{The RGDR and RGDC methods} Combining multiple rows (resp., columns) method and relaxed greedy row (resp., column) selection strategy, not like RGRK (resp., RGRCD) working randomly on one row (resp., column), we  execute all rows indexed by $U_k$ (resp., columns indexed by $V_k$) simultaneously. The  RGDR and  RGDC methods are presented as follows.

\vskip 0.15cm
\noindent{\bf The RGDR method}.  {\it Given an initial vector $\bx^{(0)} \in \Rc^{n}$ and a relaxation  parameter $\theta_1$, for $k=0,1,2,\cdots$, until the iteration sequence $\left\{\bx^{(k )}\right\}_{k=0}^{\infty}$ converges,
we first construct the index set $U_k$  according to formula \eqref{eq:RGRS}  and then update $\bx^{(k+1)}$ in accordance with
 \begin{equation*}
   \bx^{(k+1)}  = \bx^{(k)} + \frac{\bet_k^T(\bb - A\bx^{(k)})}{\BT{A^T \bet_k}} A^T \bet_k,
\end{equation*}
where $\bet_k = \sum_{i_k\in U_k} (b_{i_k} - \bal_{i_k}^T \bx^{(k)}) \bmu_{i_k}$.}

\vskip 0.15cm

\noindent{\bf The RGDC method}.  {\it Given an initial vector $\bx^{(0)} \in \Rc^{n}$ and a relaxation  parameter $\theta_2$, for $k=0,1,2,\cdots$, until the iteration sequence $\left\{\bx^{(k )}\right\}_{k=0}^{\infty}$ converges,
we first construct the index set $V_k$  according to formula \eqref{eq:RGCS}  and then update $\bx^{(k+1)}$ in accordance with
 \begin{equation*}
 \bx^{(k+1)}  = \bx^{(k)} + \frac{\bxi_k^T A^T (\bb - A\bx^{(k)})}{\BT{A\bxi_k}}\bxi_k,
\end{equation*}
where $\bxi_k = \sum_{j_k\in V_k} (\bbe_{j_k}^T(\bb-A\bx^{(k)})) \bnu_{j_k}$.}

\vskip 0.15cm
At the $k$-th iterate of RGDR, the residual vector satisfies the following recursive formula.
\begin{align*}
 \br^{(k+1)}
 & = \br^{(k)} - \frac{\bet_k^T\br^{(k)}}{\BT{A^T \bet_k}} AA^T \bet_k\\
 & = \br^{(k)} -
 \frac{\sum_{i_k\in U_k} r^{(k)}_{i_k}\bmu_{i_k}^T \br^{(k)}}
      {\sum_{i_k\in U_k} \sum_{s_k\in U_k} r^{(k)}_{i_k}  r^{(k)}_{s_k} \bmu_{i_k} AA^T \bmu_{s_k}^T}
       \sum_{i_k\in U_k}  r^{(k)}_{i_k} AA^T \bmu_{i_k}\\
 & = \br^{(k)} -
 \frac{\sum_{i_k\in U_k}\left( r^{(k)}_{i_k}\right)^2}
      {\sum_{i_k\in U_k} \sum_{s_k\in U_k}  r^{(k)}_{i_k} r^{(k)}_{s_k} \widetilde{A}_{i_k, s_k} }
       \sum_{i_k\in U_k}  r^{(k)}_{i_k} \widetilde{A}_{:,i_k},
\end{align*}
where $\widetilde{A}=AA^T$.
After that, the current vector is computed by
\begin{align*}
 \bx^{(k+1)}
  = \bx^{(k)} +
 \frac{\sum_{i_k\in U_k}\left( r^{(k)}_{i_k}\right)^2}
      {\sum_{i_k\in U_k} \sum_{s_k\in U_k} r^{(k)}_{i_k}  r^{(k)}_{s_k} \widetilde{A}_{i_k, s_k} }
       \sum_{i_k\in U_k} r^{(k)}_{i_k} \bal_{i_k}.
\end{align*}
We summarize these two computational processes as shown in Table \ref{tab:Comput_R+X_RGDR}, which only costs $(2|U_k|+1)(m+n)+\frac{1}{2}|U_k|(3|U_k|+7)$ flopping operations (flops) if we have $\widetilde{A}$ in the first place.  Furthermore, we need $m-1$ comparisons to obtain the maximum of $m$ numbers, $m$ flops to compute $\psi_{i}(\bx^{(k)})$ ($i\in [m]$), and another $2m-1$ flops to compute $\sum_{i\in [m]} \omega_i \psi_i(\bx^{(k)})$ ($i\in [m]$), thus it will cost $4m+2$ flops to determine the set $U_k$.

\begin{table}[!ht]
\centering\renewcommand\arraystretch{1.25}
    \caption{The complexities of computing   $\br^{(k+1)}$ and $\bx^{(k+1)}$ in RGDR.}
    \begin{tabular}{p{2cm}cp{6cm}cp{3cm}}
    \hline
    \multicolumn{3}{l}{{\sf Computing}~ $\br^{(k+1)}$}  \\
    \hline
   {\sf Step~1} && $g_1 = \sum_{i_k\in U_k}\left(r^{(k)}_{i_k}\right)^2$  && $2|U_k|-1$  \\
   {\sf Step~2} && $g_2 = \sum_{i_k\in U_k} \sum_{s_k\in U_k} r^{(k)}_{i_k}r^{(k)}_{s_k} \widetilde{A}_{i_k, s_k}$  &&$ \frac{3}{2}(|U_k|^2+|U_k|) $  \\
   {\sf Step~3} && $g_3 = g_1/g_2$                                 && $1$  \\
   {\sf Step~4} && ${\bm g}_1 = \sum_{i_k\in U_k} r^{(k)}_{i_k} \widetilde{A}_{:,i_k}$                                 && $m(2|U_k|-1)$  \\
   {\sf Step~5} && $\br^{(k+1)} = \br^{(k)} - g_3 \cdot {\bm g}_1$                                                     && $2m$  \\
   \hline
    \multicolumn{3}{l}{{\sf Computing}~ $\bx^{(k+1)}$}  \\
    \hline
   {\sf Step~1} && ${\bm g}_2 = \sum_{i_k\in U_k} r^{(k)}_{i_k} \bal_{i_k}$  && $n(2|U_k|-1)$  \\
   {\sf Step~2} && $\bx^{(k+1)} = \bx^{(k)} + g_3 \cdot {\bm g}_2$             && $2n$  \\
   \hline
    \end{tabular}
    \label{tab:Comput_R+X_RGDR}
\end{table}

 For RGDC, we first define an intermediate vector $\by^{(k)}=A^T\br^{(k)}$ for $k=0,1,2,\cdots$, which  satisfies the following recursive formula.
 \begin{align*}
 \by^{(k+1)}
 & = \by^{(k)} - \frac{\bxi_k^T \by^{(k)}}{\BT{A\bxi_k}} A^TA\bxi_k\\
 & = \by^{(k)} -
 \frac{\sum_{j_k\in V_k} y_{j_k}^{(k)} \bnu_{j_k}^T \by^{(k)}}
      {\sum_{j_k\in V_k}\sum_{t_k\in V_k}
      y_{j_k}^{(k)} y_{t_k}^{(k)}  \bnu_{j_k}^T A^T A \bnu_{t_k}}
      \sum_{j_k\in V_k} y_{j_k}^{(k)} A^T A\bnu_{j_k}\\
 & = \by^{(k)} -
 \frac{\sum_{j_k\in V_k} \left(y_{j_k}^{(k)}\right)^2}
      {\sum_{j_k\in V_k}\sum_{t_k\in V_k}
      y_{j_k}^{(k)} y_{t_k}^{(k)}  \widehat{A}_{j_k,t_k}}
      \sum_{j_k\in V_k} y_{j_k}^{(k)}\widehat{A}_{:,j_k},
\end{align*}
 where $\widehat{A}=A^TA$. Then the current vector is computed by

\begin{align*}
 \bx^{(k+1)}
  = \bx^{(k)} +
 \frac{\sum_{j_k\in V_k} \left(y_{j_k}^{(k)}\right)^2}
      {\sum_{j_k\in V_k}\sum_{t_k\in V_k}
      y_{j_k}^{(k)} y_{t_k}^{(k)}  \widehat{A}_{j_k,t_k}}
      \sum_{j_k\in V_k} y_{j_k}^{(k)} \bnu_{j_k}.
\end{align*}

The above two computational processes are  arranged in Table \ref{tab:Comput_R+X_RGDC}. In total, we need $(2|V_k|+1)n+\frac{1}{2}|V_k|(3|V_k|+11)$ flops to update $\bx^{(k+1)}$. Similarly, it requires another $4n+2$ flops to determine the set $V_k$.
\begin{table}[!ht]
\centering\renewcommand\arraystretch{1.25}
    \caption{ The complexities of computing   $\by^{(k+1)}$ and $\bx^{(k+1)}$ in RGDC.}
    \begin{tabular}{p{2cm}cp{7cm}cp{2.5cm}}
    \hline
    \multicolumn{3}{l}{{\sf Computing}~ $\by^{(k+1)}$}  \\
    \hline
   {\sf Step~1} && $h_1 = \sum_{j_k\in V_k} \left(y_{j_k}^{(k)}\right)^2$  && $2|V_k|-1$  \\
   {\sf Step~2} && $h_2 = \sum_{j_k\in V_k}\sum_{t_k\in V_k}y_{j_k}^{(k)} y_{t_k}^{(k)}  \widehat{A}_{j_k,t_k}$  && $ \frac{3}{2}(|V_k|^2+|V_k|) $  \\
   {\sf Step~3} && $h_3 = h_1/h_2$                                 && $1$  \\
   {\sf Step~4} && ${\bm h}_1 = \sum_{j_k\in V_k} y_{j_k}^{(k)}\widehat{A}_{:,j_k}$                                 && $n(2|V_k|-1)$  \\
   {\sf Step~5} && $\by^{(k+1)} = \by^{(k)} - h_3 \cdot {\bm h}_1$                                                     && $2n$  \\
   \hline
    \multicolumn{3}{l}{{\sf Computing}~ $\bx^{(k+1)}$}  \\
    \hline
   {\sf Step~1} && $\bx^{(k+1)} = \bx^{(k)} + h_3 \cdot \sum_{j_k\in V_k} y_{j_k}^{(k)} \bnu_{j_k}$             && $2|V_k|$  \\
   \hline
    \end{tabular}
    \label{tab:Comput_R+X_RGDC}
\end{table}

\subsection{Convergence analyses of RGDR and RGDC} In this subsection, we establish two theorems to illustrate the convergence properties of RGDR and RGDC.

\vskip 1ex
\begin{theorem}\label{thm:RGDR}
Let the linear system $A\bx=\bb$,  with the coefficient matrix $A \in \Rc^{m\times n}$ and a right-hand side $\bb \in \Rc^{m}$, be consistent. Then  the iteration sequence $\left\{\bx^{(k )}\right\}_{k=0}^{\infty}$, generated by the  RGDR  method starting from an initial guess  in the column space of $A^T$, converges to the unique least norm solution $\bx^{\ast}=A^{\dag}\bb$ determinedly. Moreover, for $k\geq 1$, the solution error obeys
\begin{equation} \label{eq:thm-RGDR}
\BT{\bx^{(k+1)} - A^\dag \bb}\leq
\left(1 - \tau_k
\frac{\sum_{i_k\in U_k} \BT{\bal_{i_k}}}{\BF{A}}
\frac{\sigma_{\min}^2(A)}{ \sigma_{\max}^2(A_{U_k,:})}   \right)
\BT{\bx^{(k)} - A^\dag \bb},
\end{equation}
 where
 $\tau_k = \theta_1 \BF{A}/\epsilon_k  + (1-\theta_1)$,
$\epsilon_k = \BF{A} - \sum_{i\in \Pi_k} \BT{\bal_i}$
and
$\Pi_k = \left\{ i\big|\psi_{i}(\bx^{(k)})=0, i\in [m]\right\}.$
\end{theorem}
\vskip 1ex
\begin{proof}
If the linear system \eqref{eq:Ax=b} is consistent, i.e., $\bb=\bb_{\textsc{r}}$ or $\bb_{\textsc{n}}=0$, then it may has infinitely
many solutions. As shown in \cite{18BW1}, the solution $\bx^{\ast} = A^\dag \bb$ is just the least Euclidean-norm solution, which satisfies
\begin{align*}
 \bx^{\ast} = \arg\min_{\bx\in\Rc^n}\BT{\bx},~s.t.~A\bx=\bb.
\end{align*}
From the Petrov-Galerkin conditions, we know that $\bet_k^T(\bb - A\bx^{(k+1)})=0$. By the Pythagorean theorem  and $\bb=\bb_{\textsc{r}}=AA^\dag \bb$, for $k\geq 1$, we have
\begin{align*}
\BT{\bx^{(k+1)} - A^\dag \bb}=\BT{ \bx^{(k)} -  A^\dag \bb} -
\frac{\big| \bet_k^T(\bb - A\bx^{(k)}) \big|^2}{\BT{A^T \bet_k}}.
\end{align*}
Since
\begin{align*}
\bet_k^T(\bb - A\bx^{(k)})
& =  \sum_{i_k\in U_k} (b_{i_k} - \bal_{i_k}^T \bx^{(k)}) \bmu_{i_k}^T  (\bb - A\bx^{(k)})\\
& =  \sum_{i_k\in U_k} \big| b_{i_k} - \bal_{i_k}^T \bx^{(k)} \big|^2 \\
& = \BT{\bet_k}
\end{align*}
and
\begin{align*}
\BT{A^T \bet_k} = \BT{A_{U_k,:}^T \widehat{\bet}_k}\leq \sigma_{\max}^2(A_{U_k,:}) \BT{\widehat{\bet}_k} = \sigma_{\max}^2(A_{U_k,:})\BT{\bet_k},
\end{align*}
where $\widehat{\bet}_k=\widehat{I}_{U_k,:}\bet_k$ with  $\widehat{I}$ being the identity matrix of size $m$, it follows that
\begin{align*}
\frac{\big| \bet_k^T(\bb - A\bx^{(k)}) \big|^2}{\BT{A^T \bet_k}}
& = \frac{\big| \bet_k^T(\bb - A\bx^{(k)}) \big|\cdot \big| \bet_k^T(\bb - A\bx^{(k)}) \big|}{\BT{A^T \bet_k}}\\
& = \frac{\sum_{i_k\in U_k} \big| b_{i_k} - \bal_{i_k}^T \bx^{(k)} \big|^2 \cdot \BT{\bet_k}}{\BT{A^T \bet_k}}\\
&\geq \frac{\sum_{i_k\in U_k} \big| b_{i_k} - \bal_{i_k}^T \bx^{(k)} \big|^2}{ \sigma_{\max}^2(A_{U_k,:})}\\
&= \frac{\sum_{i_k\in U_k} \psi_{i_k}(\bx^{(k)}) \BT{\bal_{i_k}}}{ \sigma_{\max}^2(A_{U_k,:})}.
\end{align*}
Define an auxiliary parameter
\begin{align*}
  \zeta_k:=
  \frac{\theta_1}{\BT{\bb - A\bx^{(k)}}}
  \max_{i\in [m]}\left\{ \psi_i(\bx^{(k)}) \right\} 	
  + \frac{1-\theta_1}{\BF{A}}
\end{align*}
at $\bx^{(k)}$. For any $i_k\in U_k$,  $\psi_{i_k} (\bx^{(k)}) \geq \zeta_k \BT{\bb - A\bx^{(k)}}$. Since $\bx^{(0)}, A^\dag \bb \in \Rc(A^T)$, it is easy to obtain that $\bx^{(k)} -  A^\dag \bb\in \Rc(A^T)$ by induction.  Then, we have
\begin{align*}
\frac{\big| \bet_k^T(\bb - A\bx^{(k)}) \big|^2}{\BT{A^T \bet_k}}
&\geq \frac{\sum_{i_k\in U_k} \zeta_k \BT{\bb - A\bx^{(k)}}\BT{\bal_{i_k}}}{ \sigma_{\max}^2(A_{U_k,:})}\\
&= \frac{\zeta_k \BT{\bb - A\bx^{(k)}} \sum_{i_k\in U_k}\BT{\bal_{i_k}}}{ \sigma_{\max}^2(A_{U_k,:})}\\
& = \frac{\zeta_k \BT{AA^\dag \bb - A\bx^{(k)}} \sum_{i_k\in U_k}\BT{\bal_{i_k}}}{ \sigma_{\max}^2(A_{U_k,:})}\\
&\geq \zeta_k \BF{A}
\frac{\sum_{i_k\in U_k}\BT{\bal_{i_k}}}{\BF{A}}
\frac{\sigma_{\min}^2(A)}{ \sigma_{\max}^2(A_{U_k,:})}
\BT{\bx^{(k)} -  A^\dag \bb},
\end{align*}
where the last line is from Courant-Fischer theorem.  Moreover, we get
\begin{align*}
\zeta_k \BF{A}
& = \theta_1\frac{\BF{A}}{\BT{\bb - A\bx^{(k)}}}
  \max_{i\in [m]}\left\{ \psi_i(\bx^{(k)}) \right\}
  + (1-\theta_1)\\
& = \theta_1\frac{\BF{A}}
{\sum_{i\in[m]} \psi_i(\bx^{(k)})\BT{\bal_i}}
  \max_{i\in [m]}\left\{ \psi_i(\bx^{(k)}) \right\}
  + (1-\theta_1)\\
& = \theta_1\frac{\BF{A}}
{\left(\sum_{i\in[m]}  - \sum_{i\in \Pi_k} \right)  \psi_i(\bx^{(k)})\BT{\bal_i}}
  \max_{i\in [m]}\left\{ \psi_i(\bx^{(k)}) \right\}
  + (1-\theta_1)\\
&\geq \theta_1 \frac{\BF{A}}
{\left(\sum_{i\in[m]}  - \sum_{i\in \Pi_k} \right)  \BT{\bal_i}}
  + (1-\theta_1)\\
&  =\tau_k.
\end{align*}
Then we can straightforwardly obtain the estimate in \eqref{eq:thm-RGDR}.
\hfill
\end{proof}
\vskip 1ex
In the following, we give a result to illustrate that the convergence factor of RGDR is strictly less than $1$.  Based on the fact
\begin{align*}
\BT{ \bb-A\bx^{(k)}}
& = \sum_{i \in [m]} \psi_{i}(\bx^{(k)}) \BT{\bal_{i}}\\
& = \left(\sum_{i \in [m]} - \sum_{i \in \Pi_k}\right)\psi_{i}(\bx^{(k)}) \BT{\bal_{i}} \\
& \leq \max\limits_{i\in [m]} \left\{ \psi_{i}(\bx^{(k)}) \right\} \left(\sum_{i \in [m]} - \sum_{i \in \Pi_k}\right) \BT{\bal_{i}}\\
& = \epsilon_k   \max\limits_{i\in [m]} \left\{ \psi_{i}(\bx^{(k)}) \right\},
\end{align*}
it follows that
\begin{align*}
\tau_k \frac{\sum_{i_k\in U_k} \BT{\bal_i}}{\BF{A}}
& =   \left(\theta_1 \frac{1}{\epsilon_k} + (1-\theta_1)\frac{1}{\BF{A}}\right)\sum_{i_k\in U_k} \BT{\bal_{i_k}}\\
&\leq \left(\theta_1 \frac{1}{\BT{\bb-A\bx^{(k)}}}\max\limits_{i\in [m]} \left\{ \psi_{i}(\bx^{(k)}) \right\}+ (1-\theta_1)\frac{1}{\BF{A}}\right)\sum_{i_k\in U_k} \BT{\bal_{i_k}}\\
& = \sum_{i_k\in U_k} \zeta_k \BT{\bal_{i_k}}\\
&\leq \sum_{i_k\in U_k} \frac{\psi_{i_k}(\bx^{(k)})}{\BT{\bb-A\bx^{(k)}}} \BT{\bal_{i_k}}\\
& = \sum_{i_k\in U_k} \frac{\big|  b_{i_k} - \bal_{i_k}^T \bx^{(k)}  \big|^2}{\BT{ \bb-A\bx^{(k)} }}\leq 1.
\end{align*}
Then, we have
\begin{align*}
0 \leq 1 - \tau_k
\frac{\sum_{i_k\in U_k} \BT{\bal_{i_k}}}{\BF{A}}
\frac{\sigma_{\min}^2(A)}{ \sigma_{\max}^2(A_{U_k,:})} <1.
\end{align*}

\begin{theorem}\label{thm:RGDC}
Let the linear system $A\bx=\bb$,  with the coefficient matrix $A \in \Rc^{m\times n}$ and a right-hand side $\bb \in \Rc^{m}$, be consistent or not. Then  the iteration sequence $\left\{\bx^{(k )}\right\}_{k=0}^{\infty}$, generated by the  RGDC  method starting from any initial guess, converges to the unique least-squares solution  $\bx^{\ast}=A^{\dag}\bb$ determinedly. Moreover, for $k\geq 1$, the solution error obeys
\begin{equation} \label{eq:thm-RGDC}
\BT{A\bx^{(k+1)} - AA^\dag \bb}\leq
\left(1 - \gamma_k
\frac{\sum_{j_k\in V_k} \BT{\bbe_{j_k}}}{\BF{A}}
\frac{\sigma_{\min}^2(A)}{ \sigma_{\max}^2(A_{:, V_k})}   \right)
\BT{A\bx^{(k)} - AA^\dag \bb},
\end{equation}
 where
 $\gamma_k = \theta_2 \BF{A}/\varepsilon_k + (1-\theta_2)$,
$\varepsilon_k = \BF{A} - \sum_{j\in \Omega_k} \BT{\bbe_j}$
and
$\Omega_k = \left\{ j\big|\varphi_{j}(\bx^{(k)})=0, j\in [n]\right\}.$
\end{theorem}

\begin{proof} Whether the linear system is consistent or not, we can always obtain the pseudoinverse identity $A^T \bb = A^TAA^\dag \bb$. From the Petrov-Galerkin conditions  that $\bxi_k^T A^T (\bb - A\bx^{(k+1)})=0$ and the Pythagorean theorem, for $k\geq 1$, we have
\begin{align*}
\BT{A\bx^{(k+1)} - AA^\dag \bb}=\BT{A\bx^{(k)} - AA^\dag \bb} -
\frac{\big|\bxi_k^T A^T (\bb - A\bx^{(k)})\big|^2}{\BT{A\bxi_k}}.
\end{align*}
It yields according to
\begin{align*}
\bxi_k^T A^T (\bb - A\bx^{(k)})
& = \sum_{j_k\in V_k} \bbe_{j_k}^T(\bb-A\bx^{(k)}) \bnu_{j_k}^T A^T (\bb - A\bx^{(k)})\\
& = \sum_{j_k\in V_k} \big|  \bbe_{j_k}^T(\bb-A\bx^{(k)})  \big|^2 \\
& = \BT{\bxi_k}
\end{align*}
and
\begin{align*}
\BT{A\bxi_k} = \BT{A_{:, V_k}\widetilde{\bxi}_k}\leq \sigma_{\max}^2(A_{:, V_k}) \BT{\widetilde{\bxi}_k} = \sigma_{\max}^2(A_{:, V_k})\BT{\bxi_k},
\end{align*}
where $\widetilde{\bxi}_k = \widetilde{I}_{:, V_k}^T\bxi_k$ with  $\widetilde{I}$ being the identity matrix of size $n$,  that
\begin{align*}
\frac{\big|\bxi_k^T A^T (\bb - A\bx^{(k)})\big|^2}{\BT{A\bxi_k}}
& = \frac{\big|\bxi_k^T A^T (\bb - A\bx^{(k)})\big| \cdot \big|\bxi_k^T A^T (\bb - A\bx^{(k)})\big|}{\BT{A\bxi_k}}\\
& = \frac{\sum_{j_k\in V_k} \big|  \bbe_{j_k}^T(\bb-A\bx^{(k)})  \big|^2 \cdot \BT{\bxi_k}}{\BT{A\bxi_k}}\\
&\geq \frac{\sum_{j_k\in V_k} \big|  \bbe_{j_k}^T(\bb-A\bx^{(k)})  \big|^2 }{ \sigma_{\max}^2(A_{:, V_k})}\\
& = \frac{\sum_{j_k\in V_k} \varphi_{j_k}(\bx^{(k)}) \BT{\bbe_{j_k}}  }{ \sigma_{\max}^2(A_{:, V_k})}.
\end{align*}
Define an auxiliary parameter
\begin{align*}
\delta_k :=
  \frac{\theta_2}{\BT{A^T(\bb-A\bx^{(k)})}}
  \max\limits_{j\in [n]} \left\{ \varphi_{j}(\bx^{(k)}) \right\}
  + \frac{1-\theta_2 }{\BF{A}}
  \end{align*}
at $\bx^{(k)}$. For any $j_k\in V_k$, $\varphi_{j_k}(\bx^{(k)}) \geq \delta_k   \BT{A^T(\bb-A\bx^{(k)})}$. Then, we have
\begin{align*}
\frac{\big|\bxi_k^T A^T (\bb - A\bx^{(k)})\big|^2}{\BT{A\bxi_k}}
&\geq \frac{\sum_{j_k\in V_k} \delta_k  \BT{A^T(\bb-A\bx^{(k)})} \BT{\bbe_{j_k}}  }{ \sigma_{\max}^2(A_{:, V_k})}\\
&  =  \frac{\delta_k   \BT{A^T(\bb-A\bx^{(k)})} \sum_{j_k\in V_k} \BT{\bbe_{j_k}}  }{ \sigma_{\max}^2(A_{:, V_k})}\\
&  =  \frac{\delta_k  \BT{A^TAA^\dag\bb-A^TA\bx^{(k)})} \sum_{j_k\in V_k} \BT{\bbe_{j_k}}  }{ \sigma_{\max}^2(A_{:, V_k})}\\
&\geq \delta_k \BF{A}  \frac{\sum_{j_k\in V_k} \BT{\bbe_j}}{\BF{A}}
\frac{\sigma_{\min}^2(A)}{ \sigma_{\max}^2(A_{:, V_k})}
\BT{A\bx^{(k)} - AA^\dag\bb},
\end{align*}
where the last line follows from Courant-Fischer theorem.
In addition, we know that
\begin{align*}
\delta_k \BF{A}
&\geq \theta_2\frac{\BF{A}}{\BT{A^T(\bb-A\bx^{(k)})}}
  \max\limits_{j\in [n]} \left\{ \varphi_{j}(\bx^{(k)}) \right\}
  + (1-\theta_2)\\
& =  \theta_2\frac{\BF{A}}{\sum_{j \in [n]} \varphi_{j}(\bx^{(k)}) \BT{\bbe_{j}}  }
  \max\limits_{j\in [n]} \left\{ \varphi_{j}(\bx^{(k)}) \right\}
  + (1-\theta_2)\\
& =  \theta_2\frac{\BF{A}}{\left(\sum_{j \in [n]} - \sum_{j \in \Omega_k}\right)\varphi_{j}(\bx^{(k)}) \BT{\bbe_{j}}  }
  \max\limits_{j\in [n]} \left\{ \varphi_{j}(\bx^{(k)}) \right\}
  + (1-\theta_2)\\
&\geq \theta_2\frac{\BF{A}}{\left(\sum_{j \in [n]} - \sum_{j \in \Omega_k}\right) \BT{\bbe_{j}}  }
  + (1-\theta_2)\\
& = \gamma_k{\color{blue},}
\end{align*}
which readily results in the estimate in \eqref{eq:thm-RGDC}.
\hfill
\end{proof}
\vskip 1ex
It follows from the fact
\begin{align*}
\BT{A^T(\bb-A\bx^{(k)})}
& = \sum_{j \in [n]} \varphi_{j}(\bx^{(k)}) \BT{\bbe_{j}}\\
& = \left(\sum_{j \in [n]} - \sum_{j \in \Omega_k}\right)\varphi_{j}(\bx^{(k)}) \BT{\bbe_{j}} \\
& \leq \max\limits_{j\in [n]} \left\{ \varphi_{j}(\bx^{(k)}) \right\} \left(\sum_{j \in [n]} - \sum_{j \in \Omega_k}\right) \BT{\bbe_{j}}\\
& = \varepsilon_k   \max\limits_{j\in [n]} \left\{ \varphi_{j}(\bx^{(k)}) \right\}
\end{align*}
that
\begin{align*}
\gamma_k \frac{\sum_{j_k\in V_k} \BT{\bbe_j}}{\BF{A}}
& =   \left(\theta_2 \frac{1}{\varepsilon_k} + (1-\theta_2)\frac{1}{\BF{A}}\right) \sum_{j_k\in V_k} \BT{\bbe_j}\\
&\leq \left(\theta_2 \frac{1}{\BT{A^T(\bb-A\bx^{(k)})}}\max\limits_{j\in [n]} \left\{ \varphi_{j}(\bx^{(k)}) \right\}+ (1-\theta_2)\frac{1}{\BF{A}}\right)\sum_{j_k\in V_k} \BT{\bbe_{j_k}}\\
& = \sum_{j_k\in V_k} \delta_k \BT{\bbe_{j_k}}\\
&\leq \sum_{j_k\in V_k} \frac{\varphi_{j_k}(\bx^{(k)})}{\BT{A^T(\bb-A\bx^{(k)})}} \BT{\bbe_{j_k}}\\
& = \sum_{j_k\in V_k} \frac{\big|  \bbe_{j_k}^T(\bb-A\bx^{(k)})  \big|^2}{\BT{A^T(\bb-A\bx^{(k)})}}\leq 1.
\end{align*}
Then, we have
\begin{align*}
0 \leq 1 - \gamma_k
\frac{\sum_{j_k\in V_k} \BT{\bbe_{j_k}}}{\BF{A}}
\frac{\sigma_{\min}^2(A)}{ \sigma_{\max}^2(A_{:, V_k})} <1.
\end{align*}

\begin{remark}
We remark that the factor of convergence of RGDR is smaller, uniformly with respect to the iteration step $k$, than that of RGRK given by Theorem \ref{thm:RGRK}. The reason is as follows. Since
\begin{align*}
\frac{\sum_{i_k\in U_k} \BT{\bal_{i_k}}}{\sigma_{\max}^2(A_{U_k,:})}
=\frac{\BF{A_{U_k,:}}}{\sigma_{\max}^2(A_{U_k,:})}>1
~and~
\BF{A} - \sum_{i\in \Pi_k} \BT{\bal_i} \leq \BF{A} - \min_{i\in [m]}\left\{\BT{\bal_i}\right\},
\end{align*}
i.e., $\tau_k\geq \tau$, it holds that
\begin{align*}
1 - \tau_k
\frac{\sum_{i_k\in U_k} \BT{\bal_{i_k}}}{\sigma_{\max}^2(A_{U_k,:})}
\frac{\sigma_{\min}^2(A)}{\BF{A}}
\leq
1 - \tau
\frac{\sigma_{\min}^2(A)}{\BF{A}}.
\end{align*}
In addition, it is a similar story to RGDC and RGRCD. The details here are omitted.
\end{remark}

\vskip 1em
\begin{remark}
Note that the RGDR method automatically reduces to the FDBK method \cite[Algorithm 1]{21CH} when $\theta_1=1/2$, we further generalize the FDBK method by introducing a relaxation parameter in the involved  row index set  \eqref{eq:RGRS}.  In this case, the error estimate of FDBK in \cite[Theorem 3.1]{21CH} is recovered if the index set $U_k=[m]$ or $U_k$ is the complement of $\Pi_k$.
\end{remark}

\vskip 1em

%
%

\section{Experimental results}\label{sec:ER}
To verify the effectiveness of RGDR and RGDC, two kinds of coefficient matrices are tested  as follows.
\vskip 1ex
\begin{example}\label{Ex:randn}
Consider the random matrix, generated  by the MATLAB built-in function, e.g.,
${\sf  randn}(m,n)$, it returns an $m$-by-$n$ matrix containing pseudorandom values drawn  from the standard normal distribution. In this test, the output matrix is given by $A = {\sf randn}(m,n)$.
\end{example}
\vskip 1ex
\begin{example}\label{Ex:Smatrix}
Like  \cite{19Du}, a dense  matrix is constructed  by $A = U\Sigma V^T\in \Rc^{m\times n}$,
where the entries of $U\in \Rc^{m\times r}$ and $V\in \Rc^{n\times r}$ are generated from a standard normal distribution,
and then their columns are orthonormalized. The matrix $\Sigma$ is an $r$-by-$r$ diagonal matrix whose first $r$-2 diagonal entries are uniformly distributed numbers in $ (\sigma_2, \sigma_1) $, and the last two diagonal entries are $\sigma_2$ and $\sigma_1$. The synthetic matrix is abbreviated as $ A = {\sf Smatrix} (m, n,r,\sigma_1,\sigma_2)$.
\end{example}
\vskip 1ex

 For comparison, in addition to RGRK and RGRCD, we also test several effective row and column iterative methods, such as the GBK method  \cite[Algorithm 1]{20NZ}, the randomized block Kaczmarz method (RBK) \cite[Algorithm 1.1]{14NT}, the randomized block CD (RBCD) method \cite[Algorithm 2]{15NZZ}, and the accelerated max-distance CD (AMDCD) method \cite[Algorithm 4]{20ZG}. These methods are briefly described below.

RBK and GBK are two special block Kaczmarz methods, which all admit  the following iteration scheme. For $k=0,1,2,\cdots$,
\begin{align*}
  \bx^{(k+1)}  = \bx^{(k)} +
  A_{I_k, :}^{\dag} (\bb_{I_k} - A_{I_k, :}\bx^{(k)}),
\end{align*}
where $I_k$ is a row index set. This process is identical to project the current iterate point $\bx^{(k)}$ orthogonally onto the solution space of $\{\bx|A_{I_k, :}\bx = \bb_{I_k}\}$. In RBK, $I_k$ is  selected uniformly at random from a partition of the row indices of $A$, which is defined by $\widehat{T} = \{\widehat{\tau}_0, \widehat{\tau}_1,\cdots, \widehat{\tau}_r \}$, where
$\widehat{\tau}_i = \{i\widehat{s}+1,i\widehat{s}+2,\cdots,i\widehat{s}+\widehat{s}\}$
with $i=0,1,\cdots,r-1$, $\widehat{\tau}_r = \left\{ r\widehat{s}+1,r\widehat{s}+2,\cdots,m\right\}$, and $\widehat{s}$ is the block size. In GBK, $I_k$ is determined by a greedy selection strategy. That is,
\begin{equation*}
  I_k=\left\{i_k \Big| \psi_{i_k}(\bx^{(k)}) \geq
  \eta_1\cdot \max_{i\in [m]}\left\{ \psi_i(\bx^{(k)}) \right\} , i_k\in [m]\right\}.
\end{equation*}

In RBCD, a partition of the column indices of $A$ is needed. It is achieved by $\widetilde{T} = \{\widetilde{\tau}_0, \widetilde{\tau}_1,\cdots, \widetilde{\tau}_c \}$, where
$\widetilde{\tau}_j = \{j\widetilde{s}+1,j\widetilde{s}+2,\cdots,j\widetilde{s}+\widetilde{s}\}$
with $j=0,1,\cdots,c-1$, $\widetilde{\tau}_c = \left\{ c\widetilde{s}+1,c\widetilde{s}+2,\cdots,n\right\}$, and $\widetilde{s}$ is the block size. After picking the block $\widetilde{\tau}_k$ uniformly at random from $\widetilde{T}$, the RBCD iterate is given by
\begin{align*}
  \bx^{(k+1)}_{\widetilde{\tau}_k}  = \bx^{(k)}_{\widetilde{\tau}_k} + \bw^{(k)}
\end{align*}
for $k=0,1,2,\cdots$, where $\bw^{(k)}=A_{:,{\widetilde{\tau}_k}}^{\dag} \bz^{(k)}$ and $\bz^{(k+1)}=\bz^{(k)}-A_{:,{\widetilde{\tau}_k}}\bw^{(k)}$ with $\bx^{(0)}={\bm 0}$ and $\bz^{(0)}={\bm b}$.

From the max-distance selection strategy, i.e., for $k=0,1,2,\cdots$,
\begin{equation*}
  J_k=\left\{j_k \Big|
  |D_{\max} - D_{j_k}| \leq \eta_2, j_k\in [n]\right\}~{\rm with}~
  D_j = \frac{|\bbe_j^T(\bb - A\bx^{(k)})|}{\bT{\bbe_j}}~{\rm and}~
  D_{\max} = \max_{j\in[n]} \left\{D_j\right\},
\end{equation*}
AMDCD uses a pseudoinverse-free method to update the next iterate. That is,
\begin{align*}
  \bx^{(k+1)} = \bx^{(k)}  + \sum_{j_k \in J_k} \frac{y_{j_k}^{(k)}}{\BT{\bbe_{j_k}}}\bnu_{j_k}
  ~{\rm with}~
  \by^{(k+1)} = \by^{(k)} - \sum_{j_k \in J_k} \frac{y_{j_k}^{(k)}}{\BT{\bbe_{j_k}}}\widehat{A}_{:,j_k}.
\end{align*}

Unless otherwise specified, at each iteration $k$, we keep track of the number of iteration steps (denoted by {\sf IT}) and the CPU time in seconds (denoted by {\sf CPU}).  For the randomized methods, e.g., RGRK and RGRCD, each {\sf IT} and {\sf CPU} are the arithmetic mean of the results obtained by repeatedly running the corresponding method $30$ times. All numerical tests are performed  on a Founder desktop PC with Intel(R) Core(TM) i5-7500 CPU 3.40 GHz.

Let the relative solution error (denoted by {\sf RSE}) be
${\sf RSE} =  \bT{\bx^{(k)} - \bx^{\ast} } \big /\bT{\bx^{(0)} - \bx^{\ast} }$ at $\bx^{(k)}$, where we set $\bx^{(0)}=\bm{0}$.  The experiments are terminated  once {\sf RSE} is less than $10^{-4}$, or {\sf IT} exceeds $10^{6}$.
 To make the implementations of RGDR and RGDC more efficient, we try to avoid using for-loop structure as far as possible at each iteration. In addition, we execute RGDR and RGDC without explicitly forming the matrices $\widetilde{A}$ and $\widehat{A}$. Note that RGDR is equivalent to  FDBK  \cite{21CH} when the relaxation parameter $\theta_1 = 0.5$.  We take the parameters $\eta_1=0.5$ in GBK and $\eta_2=0.1$ in AMDCD as the same values as that used in \cite{20NZ} and \cite{20ZG}, respectively. In RBK and RBCD, the block size is set by $\widehat{s} = \widetilde{s} = 100$. To avoid calculating the Moore-Penrose inverse of $A_{I_k, :}$ in RBK and GBK and $A_{:,{\widetilde{\tau}_k}}$ in RBCD, we use the CGLS algorithm \cite{Saad2000}  to solve the least-squares problem. In the following
 the symbols RGDR($\theta$) and RGRK($\theta$) respectively represent the RGDR and RGRK methods with parameter $\theta$. This notation also applies to RGDC and RGRCD. The numerical results for solving the consistent and  inconsistent linear systems are reported.

\vskip 1ex
{\textbf{Case~1: the overdetermined and consistent problems.}}  To ensure that the linear system \eqref{eq:Ax=b} is consistent, the solution vector $\bx^{\ast}\in \Rc^n$ with entries randomly generated by the
MATLAB function, e.g., {\sf randn(n,1)}, is formed at first, and  the right-hand-side is then set to be $\bb = A\bx^{\ast}$.
The numerical results are reported in Tables \ref{tab:randn-GRDRvsRGRK} and \ref{tab:randn-GRDCvsRGRCD}.  It yields the following observations.

Both RGDR and RGDC can converge to the  solution of the overdetermined and consistent linear system. RGDR (resp., RGDC) vastly outperforms RGRK (resp., RGRCD) in terms of both iteration counts and CPU times if they choose a same relaxation parameter. We also find that RGDR (resp., RGDC) works better than GBK and RBK (resp., AMDCD and RBCD)  in terms of CPU times.

 We plot the curves of CPU times versus the relaxation parameter $\theta_1$ and $\theta_2$ for RGDR and RGDC in Figure \ref{fig:THETAvsCPUrandn}. In this figure, we see that the number of iteration steps is increasing when the sizes of $A$ are fixed but the relaxation parameter is grown, while the CPU time is decreasing firstly and then increasing. This is because the smaller the relaxation parameter is, the larger the size of $U_k$ or $V_k$ is. When $\theta_1=0.7$,  RGDR  is more effective than FDBK in terms of CPU times in all cases, which means that RGDR can further improve the efficiency of  FDBK.

\begin{table}[!ht]
 \normalsize
    \caption{ The numerical results of IT and CPU for RGRK \cite{18BW2}, RBK \cite{14NT}, GBK \cite{20NZ}, and RGDR, where the coefficient matrix is  from Example \ref{Ex:randn} with $ A = {\sf randn} (m, 300)$ and $\bb = A\bx^{\ast}$.}
    \centering
    \begin{tabular}{l l c c c c c }
    \hline
       &$m$  & $5000$ & $8000$ & $10000$ & $12000$ & $15000$ \\
    \hline
    RBK & {\sf RSE }&$9.28\times 10^{-5}$&$9.09\times 10^{-5}$&$9.91\times 10^{-5}$&$9.94\times 10^{-5}$&$9.65\times 10^{-5}$ \\
        & {\sf IT } &46.2&46.1&46.8&46.7&46.0\\
        & {\sf CPU} &0.2667&0.2822&0.3010&0.3094& 0.3713\\
    \hline
    GBK & {\sf RSE } &$8.37\times 10^{-5}$&$9.26\times 10^{-5}$&$8.21\times 10^{-5}$&$8.41\times 10^{-5}$&$6.45\times 10^{-5}$\\
        & {\sf IT } &20.3&19.4&18.6&15.9&14.2\\
        & {\sf CPU} &0.0780&0.0639&0.0760&0.0800&0.0878\\
    \hline
RGRK(0.3)& {\sf RSE }&$9.99\times 10^{-5}$&$9.95\times 10^{-5}$&$9.86\times 10^{-5}$&$9.90\times 10^{-5}$&$9.97\times 10^{-5}$\\
         & {\sf IT } &867.0&818.6&778.6&774.3&753.7\\
         & {\sf CPU} &0.1319&0.1482&0.1645&0.1906&0.2142\\
RGRK(0.5)& {\sf RSE } &$9.99\times 10^{-5}$&$9.99\times 10^{-5}$&$9.88\times 10^{-5}$&$9.95\times 10^{-5}$&$9.88\times 10^{-5}$\\
         & {\sf IT } &678.9&614.7&591.4&581.6&555.0\\
         & {\sf CPU} &0.1016&0.1098&0.1196&0.1551&0.0967 \\
RGRK(0.7)& {\sf RSE } &$9.87\times 10^{-5}$&$9.90\times 10^{-5}$&$9.82\times 10^{-5}$&$9.88\times 10^{-5}$&$9.89\times 10^{-5}$\\
         & {\sf IT } &607.8&532.6&506.7&486.9&466.1\\
         & {\sf CPU} &0.0811&0.1038&0.1014&0.1353&0.0867\\
RGRK(0.9)& {\sf RSE } &$9.97\times 10^{-5}$&$9.96\times 10^{-5}$&$9.97\times 10^{-5}$&$9.99\times 10^{-5}$&$9.92\times 10^{-5}$\\
         & {\sf IT } &573.2&498.9&470.7&452.0&431.5\\
         & {\sf CPU} &0.0683&0.0838&0.0937&0.1037&0.0623\\
   \hline
RGDR(0.3)& {\sf RSE }&$7.78\times 10^{-5}$&$5.25\times 10^{-5}$&$7.48\times 10^{-5}$&$8.23\times 10^{-5}$&$6.40\times 10^{-5}$\\
         & {\sf IT } &15&12&12&11&9 \\
         & {\sf CPU} &0.0906&0.1280&0.1361&0.1721&0.2089 \\
RGDR(0.5)& {\sf RSE }&$9.73\times 10^{-5}$&$9.68\times 10^{-5}$&$8.63\times 10^{-5}$&$6.45\times 10^{-5}$&$7.30\times 10^{-5}$\\
         & {\sf IT } &29&23&23&21&19\\
         & {\sf CPU} &0.0517&0.0687&0.0724&0.1034&0.0711\\
RGDR(0.7)& {\sf RSE }&$8.59\times 10^{-5}$&$9.72\times 10^{-5}$&$7.39\times 10^{-5}$&$8.84\times 10^{-5}$&$9.19\times 10^{-5}$\\
         & {\sf IT } &66&56&54&54&45\\
         & {\sf CPU} &0.0431&0.0499&0.0548&0.0740&0.0419\\
RGDR(0.9)& {\sf RSE }&$9.77\times 10^{-5}$&$9.99\times 10^{-5}$&$9.55\times 10^{-5}$&$9.86\times 10^{-5}$&$9.69\times 10^{-5}$\\
         & {\sf IT } &219&205&182&182&160\\
         & {\sf CPU} &0.0562&0.0597&0.0660&0.0715&0.0364\\
    \hline
    \end{tabular}
    \label{tab:randn-GRDRvsRGRK}
\end{table}

\begin{table}[!ht]
 \normalsize
    \caption{ The numerical results of IT and CPU for RGRCD \cite{20ZG}, RBCD \cite{15NZZ}, AMDCD \cite{20ZG}, and RGDC, where the coefficient matrix is  from Example \ref{Ex:randn} with $ A = {\sf randn} (m, 300)$ and $\bb = A\bx^{\ast}$.}
    \centering
    \begin{tabular}{l l c c c c c }
    \hline
       &$m$  & $5000$ & $8000$ & $10000$ & $12000$ & $15000$ \\
    \hline
    RBCD & {\sf RSE }&$6.15\times 10^{-5}$&$6.10\times 10^{-5}$&$4.08\times 10^{-5}$&$8.66\times 10^{-5}$&$9.41\times 10^{-5}$ \\
        & {\sf IT } &61.2&70.8&79.0&61.2&54.2\\
        & {\sf CPU} &0.4839&0.7934&1.3666&1.2402&2.6673\\
    \hline
    AMDCD & {\sf RSE } &$5.81\times 10^{-5}$&$7.17\times 10^{-5}$&$5.32\times 10^{-5}$&$2.92\times 10^{-5}$&$1.96\times 10^{-5}$ \\
        & {\sf IT } &502.1&488.9&474.2&465.0&455.9\\
        & {\sf CPU} &0.0955&0.0883&0.0863&0.0803&0.0810\\
    \hline
RGRCD(0.3)& {\sf RSE }&$9.97\times 10^{-5}$&$9.97\times 10^{-5}$&$9.92\times 10^{-5}$&$9.88\times 10^{-5}$&$9.96\times 10^{-5}$ \\
         & {\sf IT } &1461.1&1214.9&1158.0&1128.8& 1065.2\\
         & {\sf CPU} &0.0247&0.0220&0.0230&0.0200&0.0197\\
RGRCD(0.5)& {\sf RSE }&$9.99\times 10^{-5}$&$9.97\times 10^{-5}$&$9.97\times 10^{-5}$&$9.95\times 10^{-5}$&$\times 10^{-5}$ \\
         & {\sf IT } &1454.5&1186.2&1152.7&1119.0&1061.3\\
         & {\sf CPU} &0.0245&0.0211&0.0221&0.0186& 0.0190\\
RGRCD(0.7)& {\sf RSE }&$9.95\times 10^{-5}$&$9.96\times 10^{-5}$&$\times 10^{-5}$&$9.95\times 10^{-5}$&$9.99\times 10^{-5}$ \\
         & {\sf IT } &1453.4&1178.4&1149.3&1112.5&1053.1\\
         & {\sf CPU} &0.0230&0.0197&0.0214&0.0185&0.0170\\
RGRCD(0.9)& {\sf RSE }&$9.98\times 10^{-5}$&$9.95\times 10^{-5}$&$9.97\times 10^{-5}$&$9.99\times 10^{-5}$&$9.98\times 10^{-5}$ \\
         & {\sf IT } &1450.2&1167.4&1141.6&1104.7&1055.3\\
         & {\sf CPU} &0.0216&0.0167&0.0200&0.0171&0.0165\\
   \hline
RGDC(0.3)& {\sf RSE }&$8.78\times 10^{-5}$&$8.83\times 10^{-5}$&$8.71\times 10^{-5}$&$9.12\times 10^{-5}$&$9.14\times 10^{-5}$ \\
         & {\sf IT } &31&27&25&24&23\\
         & {\sf CPU} &0.0060&0.0076&0.0070&0.0060&0.0073\\
RGDC(0.5)& {\sf RSE }&$9.27\times 10^{-5}$&$9.58\times 10^{-5}$&$8.31\times 10^{-5}$&$8.66\times 10^{-5}$&$9.53\times 10^{-5}$ \\
         & {\sf IT } &52&45&43&40&37\\
         & {\sf CPU} &0.0053&0.0053&0.0051&0.0041&0.0043\\
RGDC(0.7)& {\sf RSE }&$9.68\times 10^{-5}$&$9.73\times 10^{-5}$&$9.10\times 10^{-5}$&$9.59\times 10^{-5}$&$9.88\times 10^{-5}$ \\
         & {\sf IT } &97&80&76&73&68\\
         & {\sf CPU} &0.0041&0.0051&0.0032&0.0040&0.0040\\
RGDC(0.9)& {\sf RSE }&$9.90\times 10^{-5}$&$9.93\times 10^{-5}$&$9.88\times 10^{-5}$&$9.49\times 10^{-5}$&$9.71\times 10^{-5}$ \\
         & {\sf IT } &278&232&225&206&194\\
         & {\sf CPU} &0.0076&0.0093&0.0068&0.0073&0.0064\\
    \hline
    \end{tabular}
    \label{tab:randn-GRDCvsRGRCD}
\end{table}

\begin{figure}[!ht]
\setlength{\abovecaptionskip}{0pt}
\setlength{\belowcaptionskip}{0pt}
\centering
	\subfigure[RGDR]{\includegraphics[width=0.5\textwidth]{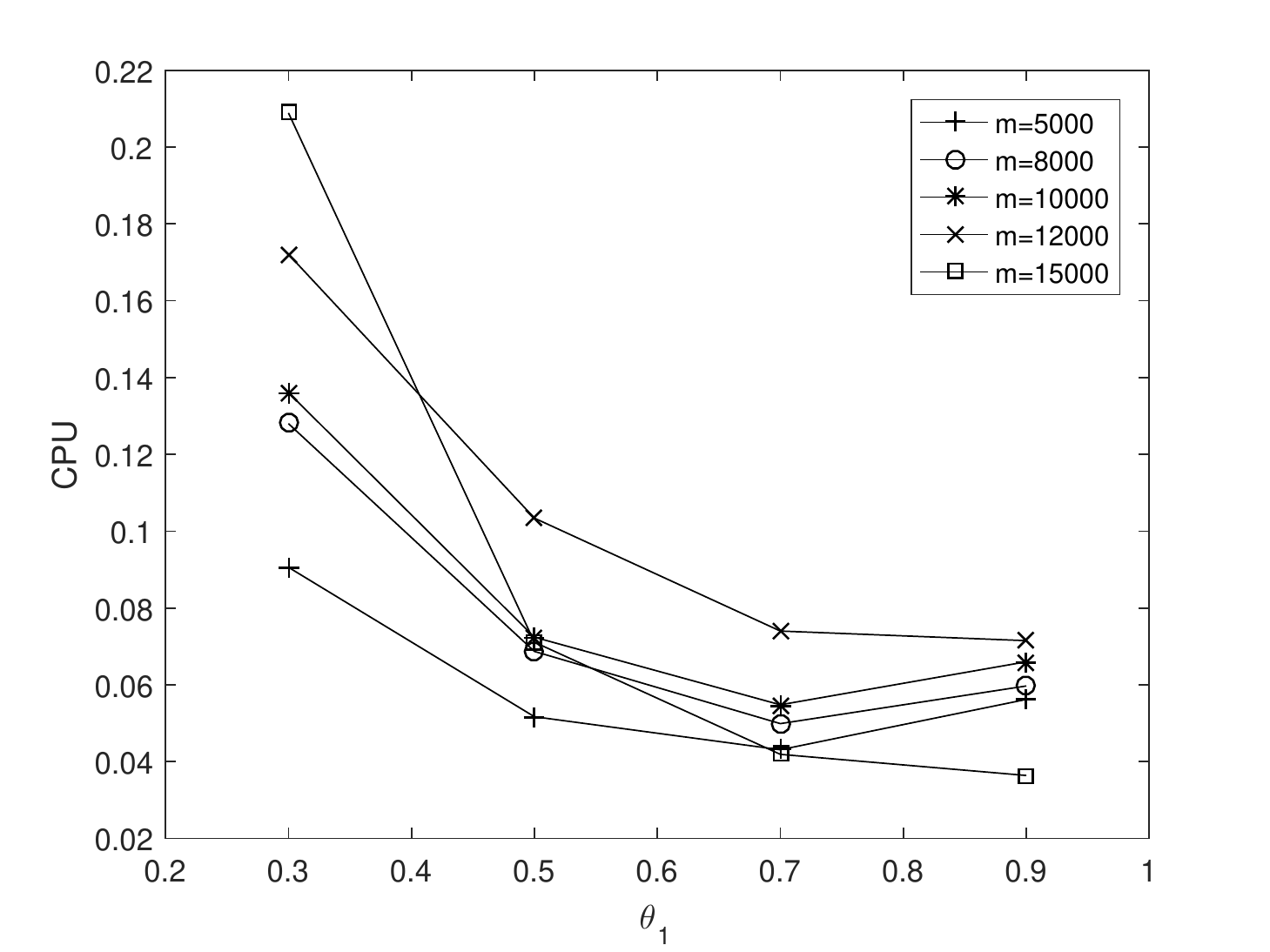}}
 \hspace{-0.5cm}
 	\subfigure[RGDC]{\includegraphics[width=0.5\textwidth]{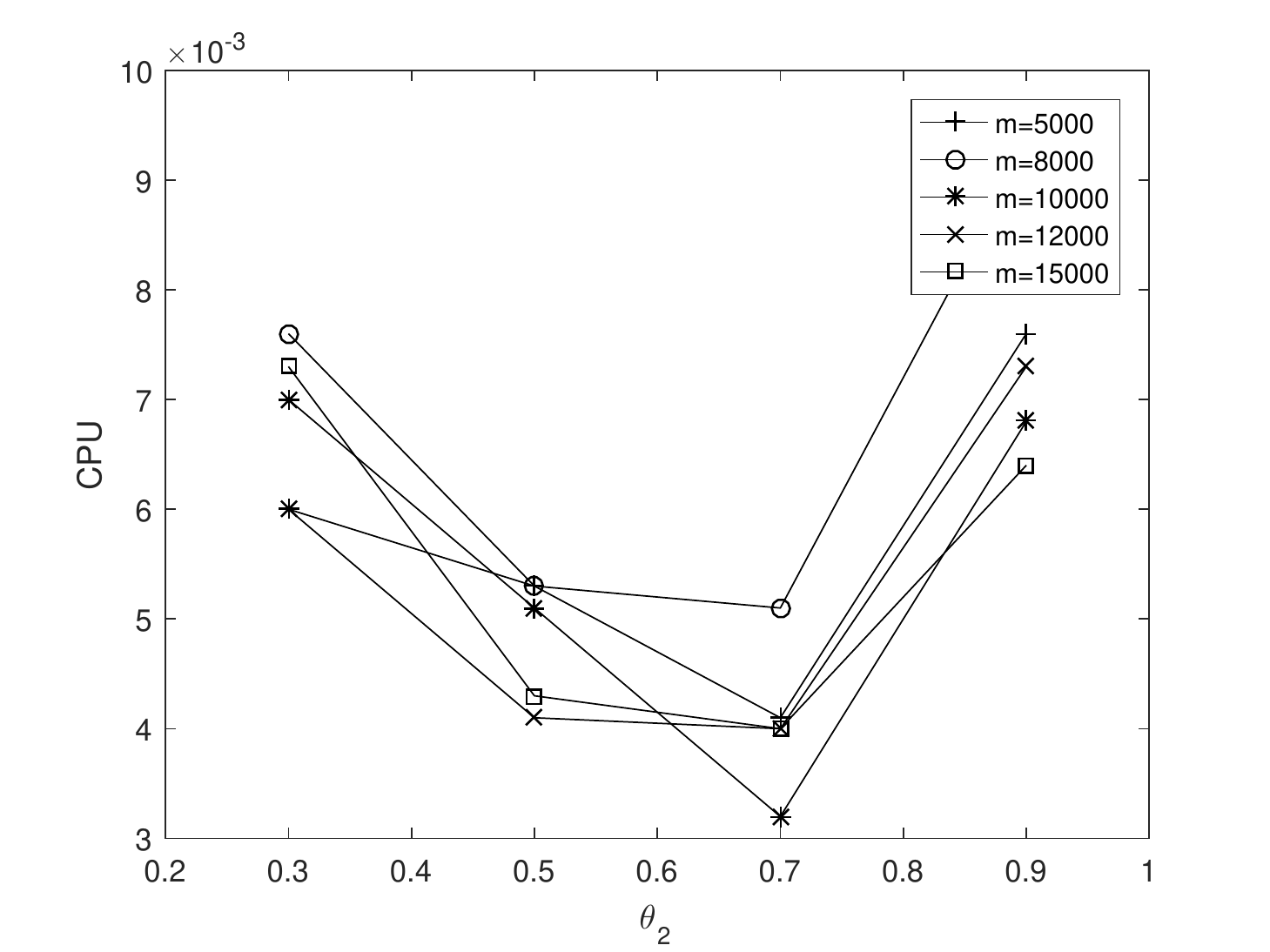}}
\caption{ The curves of relaxation parameter  versus  CPU for RGDR (left) and RGDC (right), where the coefficient matrix is  from Example \ref{Ex:randn} with $ A = {\sf randn} (m, 300)$ and $\bb = A\bx^{\ast}$.}
  \label{fig:THETAvsCPUrandn}
\end{figure}

\vskip 1ex
{\textbf{Case~2: the overdetermined and inconsistent problems.}}  For the inconsistent case, the  noisy right-hand side is set by $ \bb  = A \bx^{\ast} + \bde\bb$, where $\bde\bb$ belongs to the null space of $A^T$ and is generated by the   MATLAB function, e.g.,  {\sf null}.  We plot the curves of the relative solution error  versus  CPU time for RGDC($\theta_2$) and RGRCD($\theta_2$) with $\theta_2=0.3$, $0.5$, $0.7$ and $0.9$, as shown in Figures \ref{fig:RGDCvsRGRCDnoisy+randn} and \ref{fig:RGDCvsRGRCDnoisy+smatrix}. We observe that  RGDC converges linearly to the least-squares solution when the system is inconsistent, and  the relative solution error of RGDC is decaying much more quickly than that of RGRCD when the CPU time is increasing  if they choose a same relaxation parameter.

\begin{figure}[!ht]
\setlength{\abovecaptionskip}{0pt}
\setlength{\belowcaptionskip}{0pt}
\centering
	\subfigure[$\theta_2=0.3$]{\includegraphics[width=0.5\textwidth]{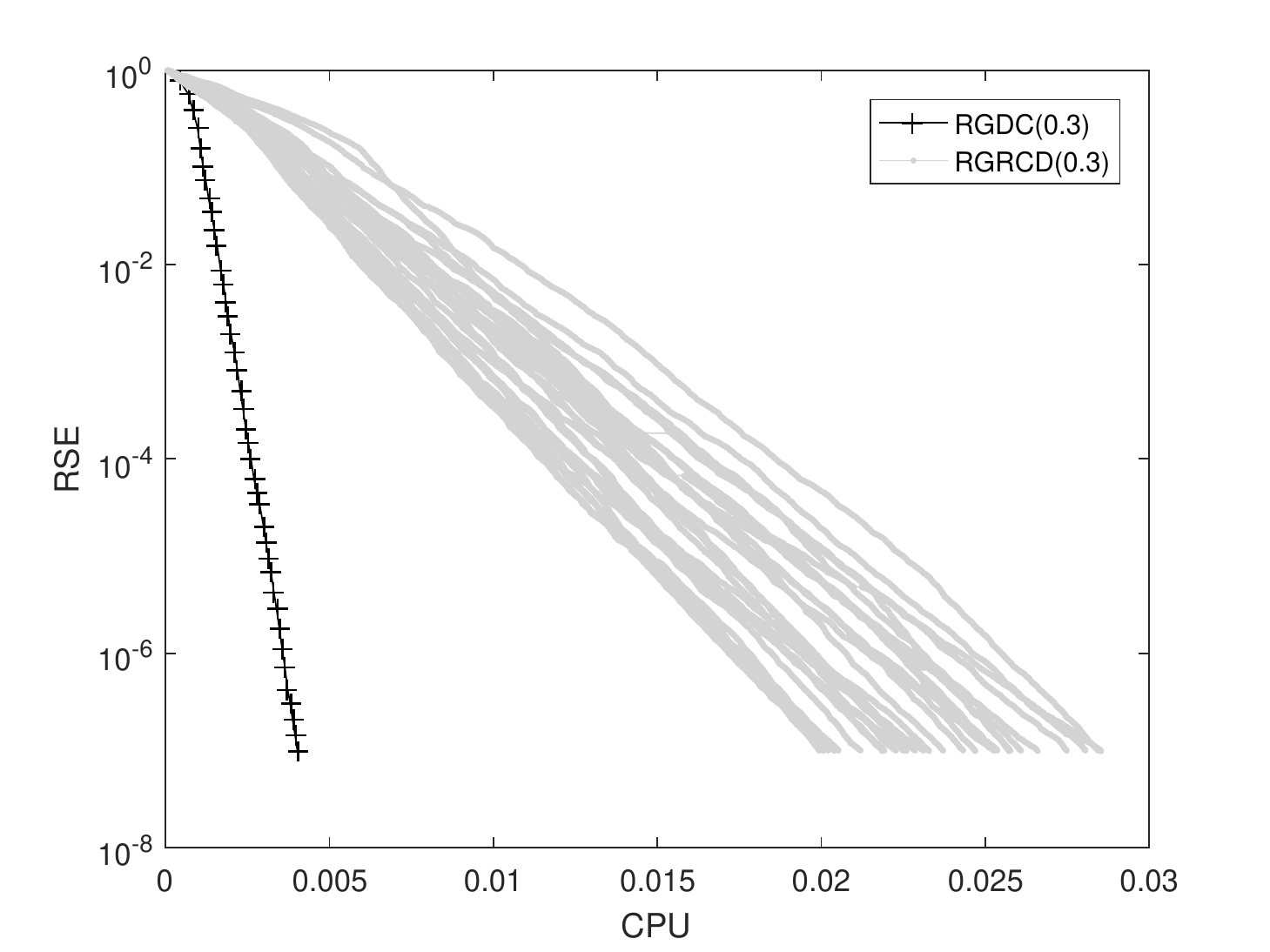}}
 \hspace{-0.5cm}
 	\subfigure[$\theta_2=0.5$]{\includegraphics[width=0.5\textwidth]{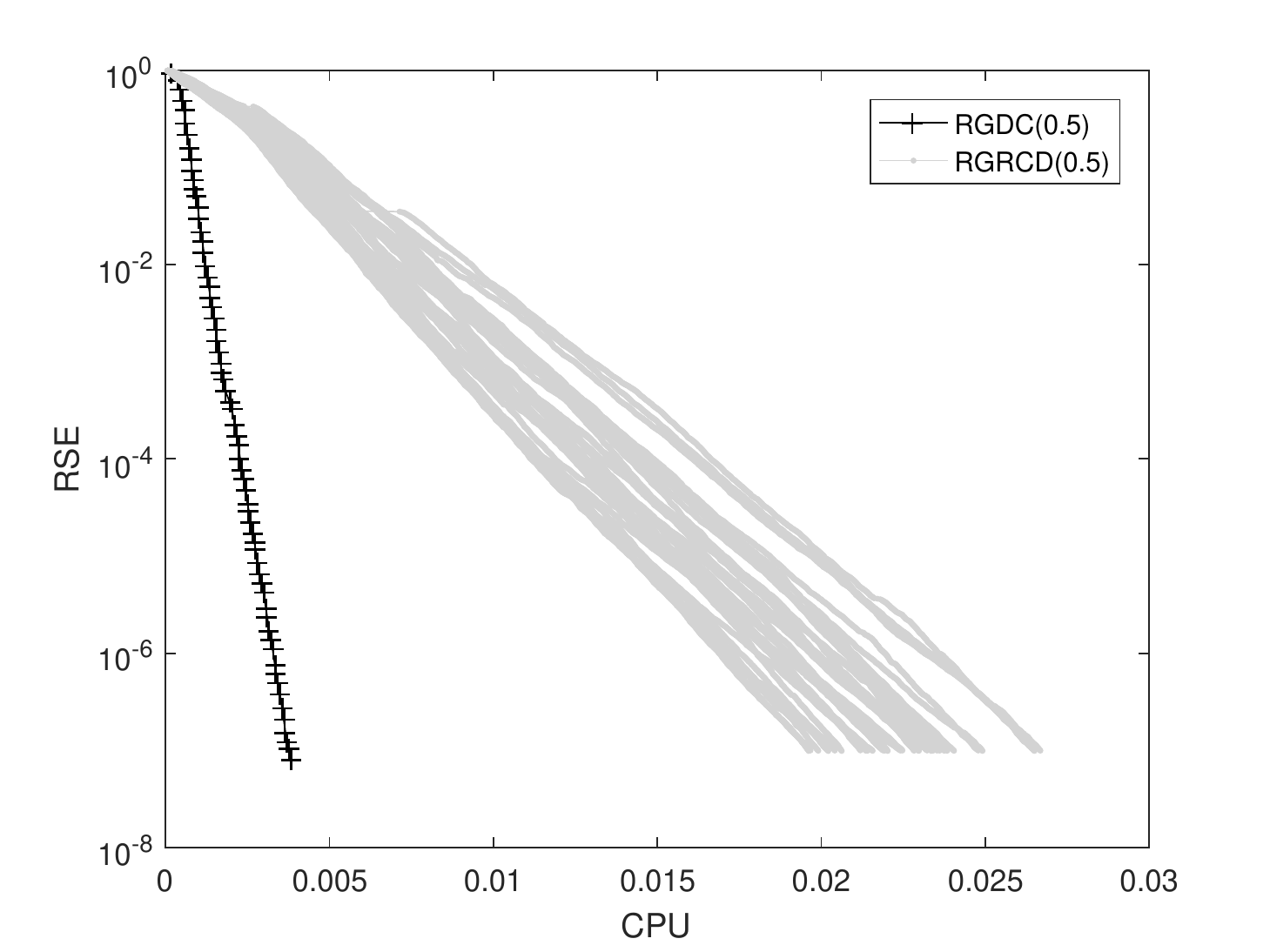}}
 \hspace{5cm}
 	\subfigure[$\theta_2=0.7$]{\includegraphics[width=0.5\textwidth]{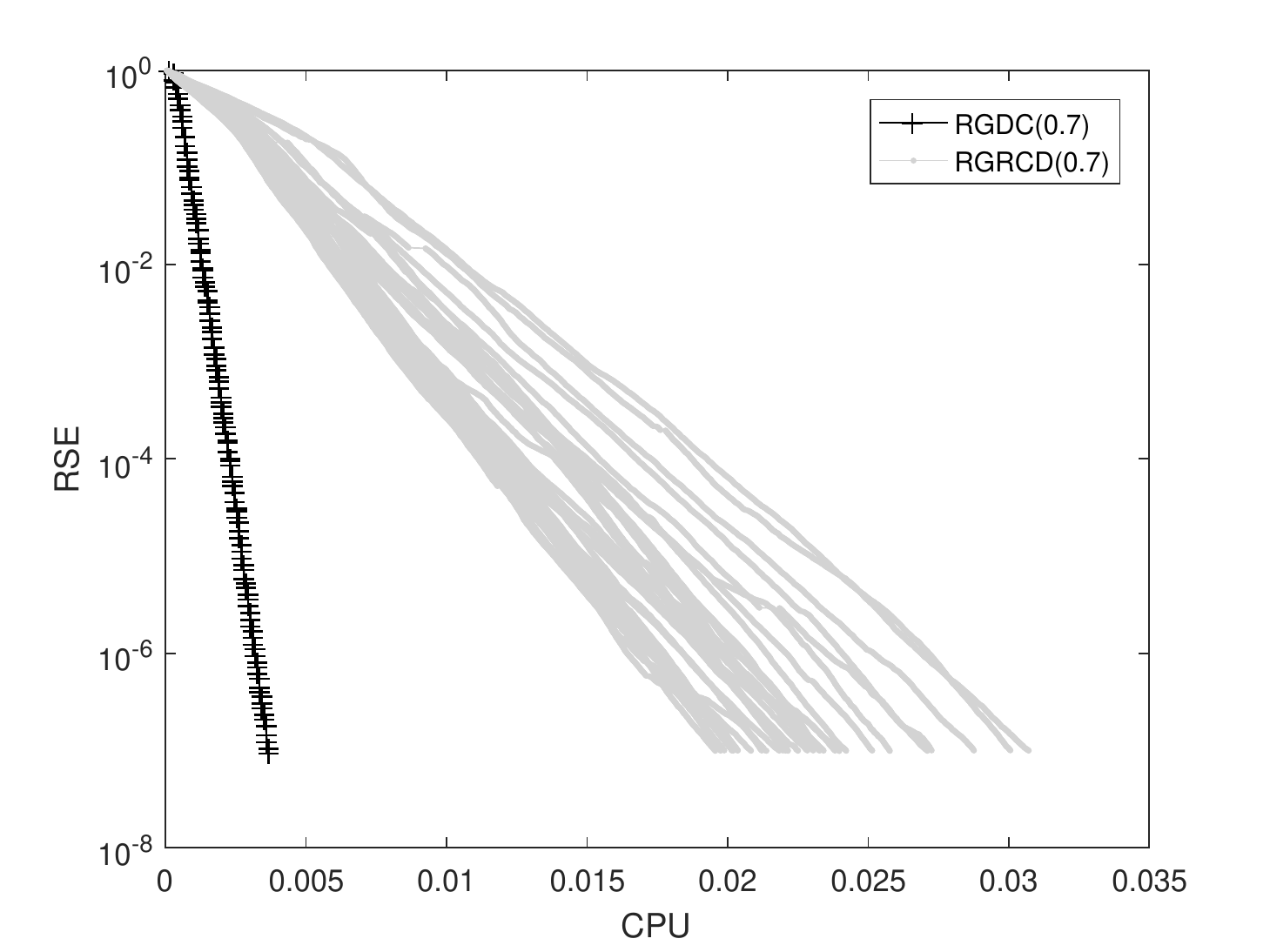}}
 \hspace{-0.5cm}
 	\subfigure[$\theta_2=0.9$]{\includegraphics[width=0.5\textwidth]{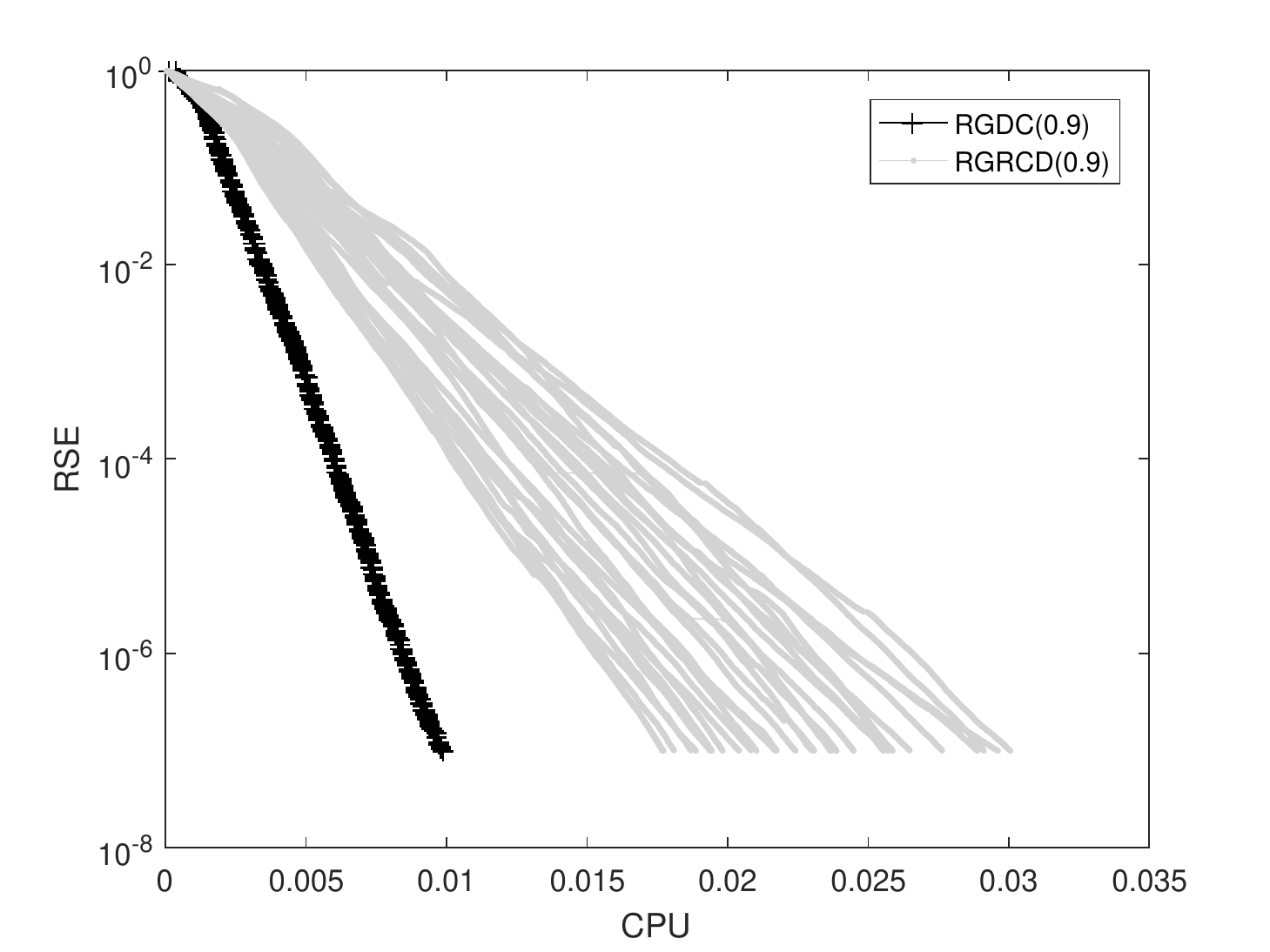}}
\caption{ The convergence curves of RSE  versus  CPU for RGDC($\theta_2$) and RGRCD($\theta_2$) with $\theta_2=0.3$, $0.5$, $0.7$ and $0.9$, where the coefficient matrix is  from Example \ref{Ex:Smatrix} with $A = {\sf Smatrix} (10000, 300, 300, 1.25, 1)$ and $\bb  = A \bx^{\ast} + \bde\bb$.}
  \label{fig:RGDCvsRGRCDnoisy+randn}
\end{figure}

\begin{figure}[!ht]
\setlength{\abovecaptionskip}{0pt}
\setlength{\belowcaptionskip}{0pt}
\centering
	\subfigure[$\theta_2=0.3$]{\includegraphics[width=0.5\textwidth]{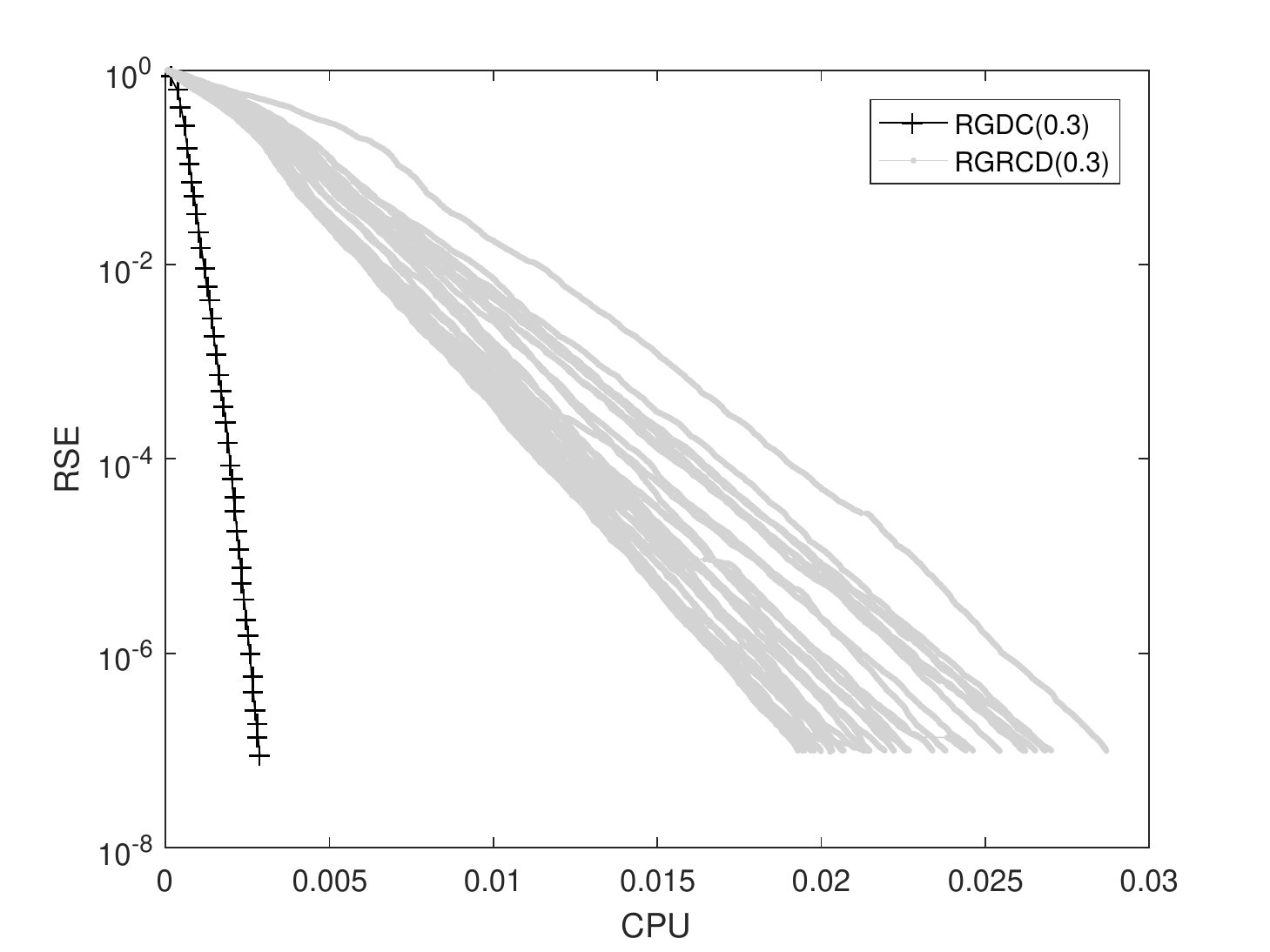}}
 \hspace{-0.5cm}
 	\subfigure[$\theta_2=0.5$]{\includegraphics[width=0.5\textwidth]{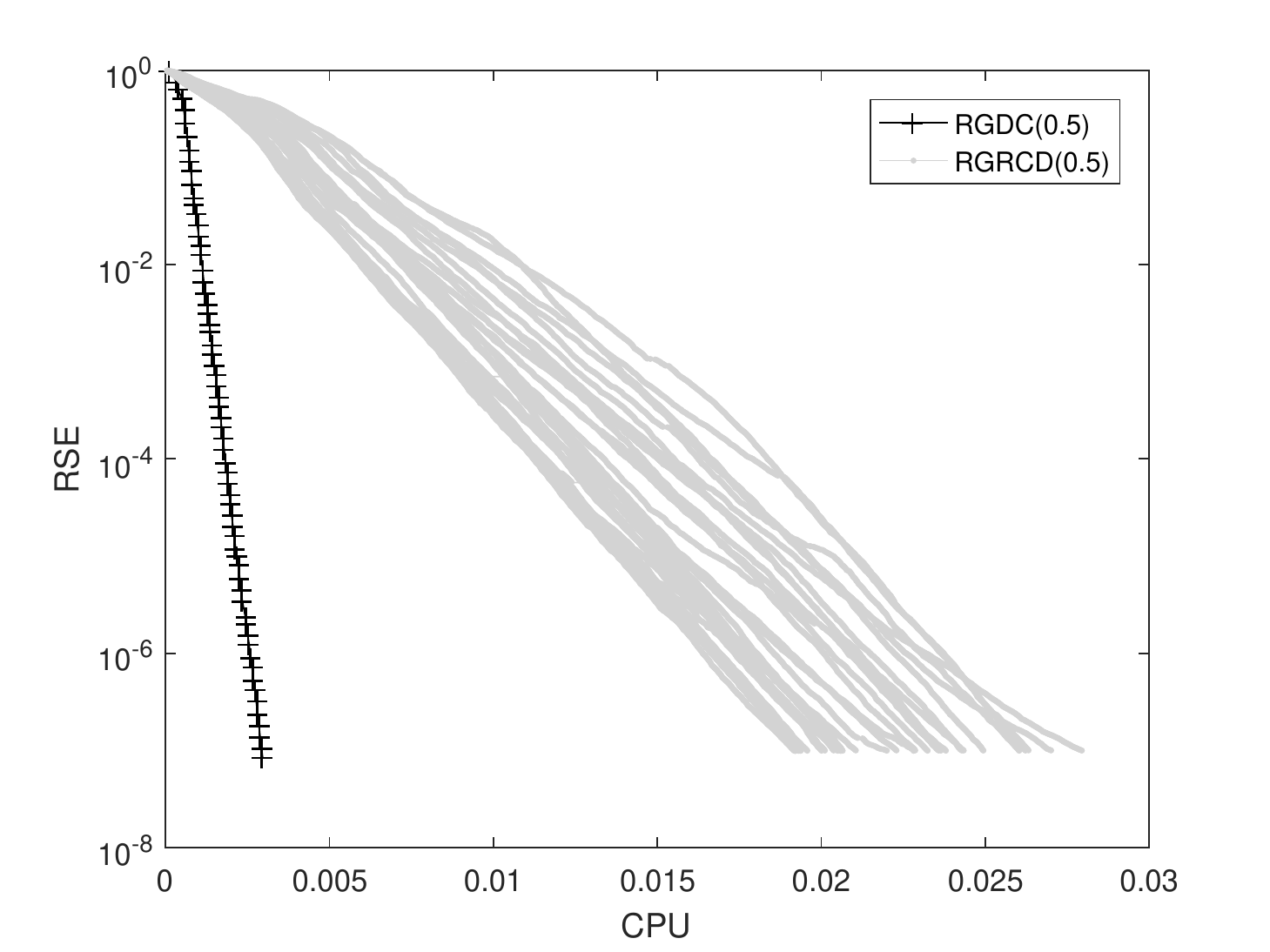}}
 \hspace{5cm}
 	\subfigure[$\theta_2=0.7$]{\includegraphics[width=0.5\textwidth]{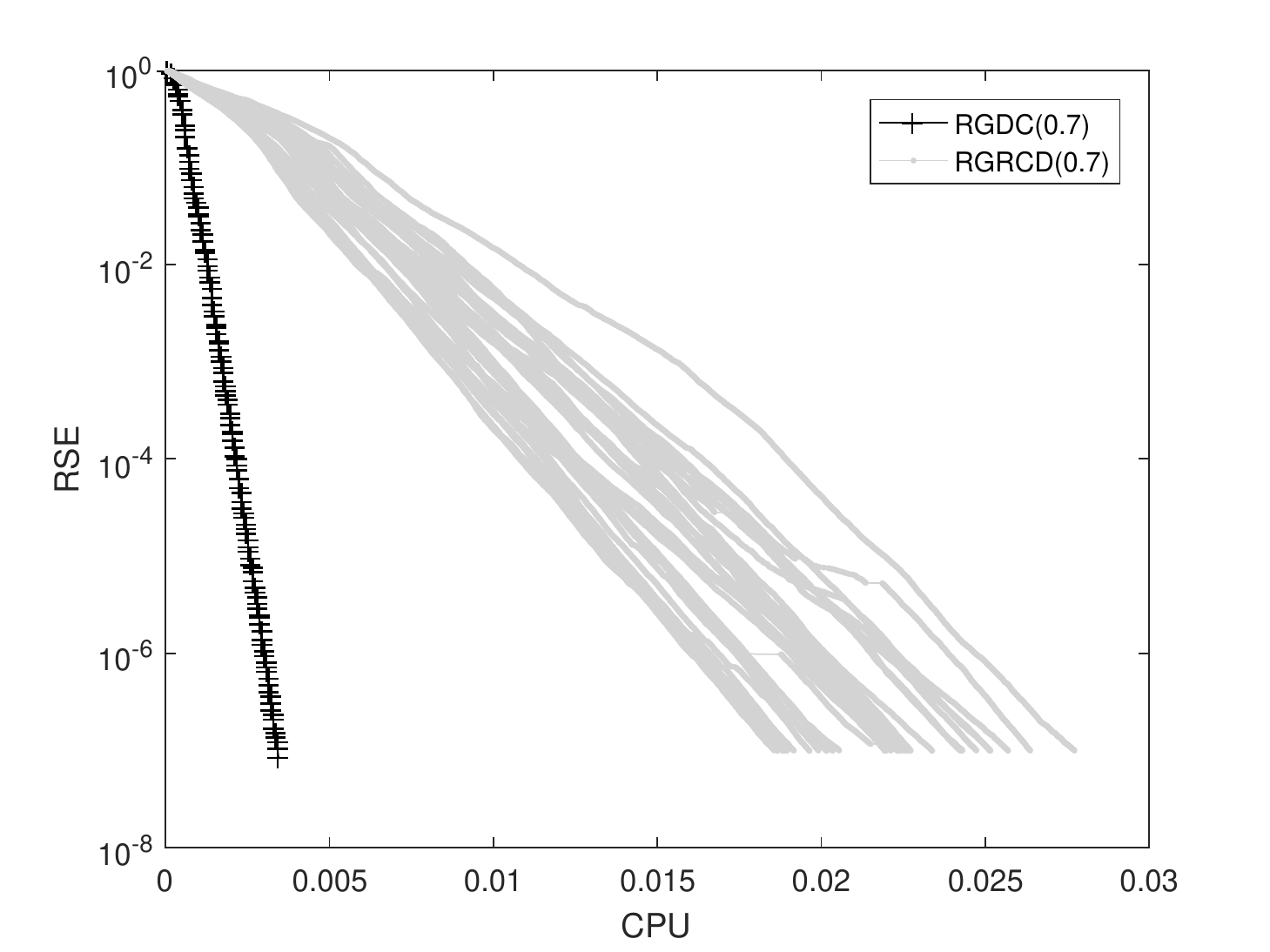}}
 \hspace{-0.5cm}
 	\subfigure[$\theta_2=0.9$]{\includegraphics[width=0.5\textwidth]{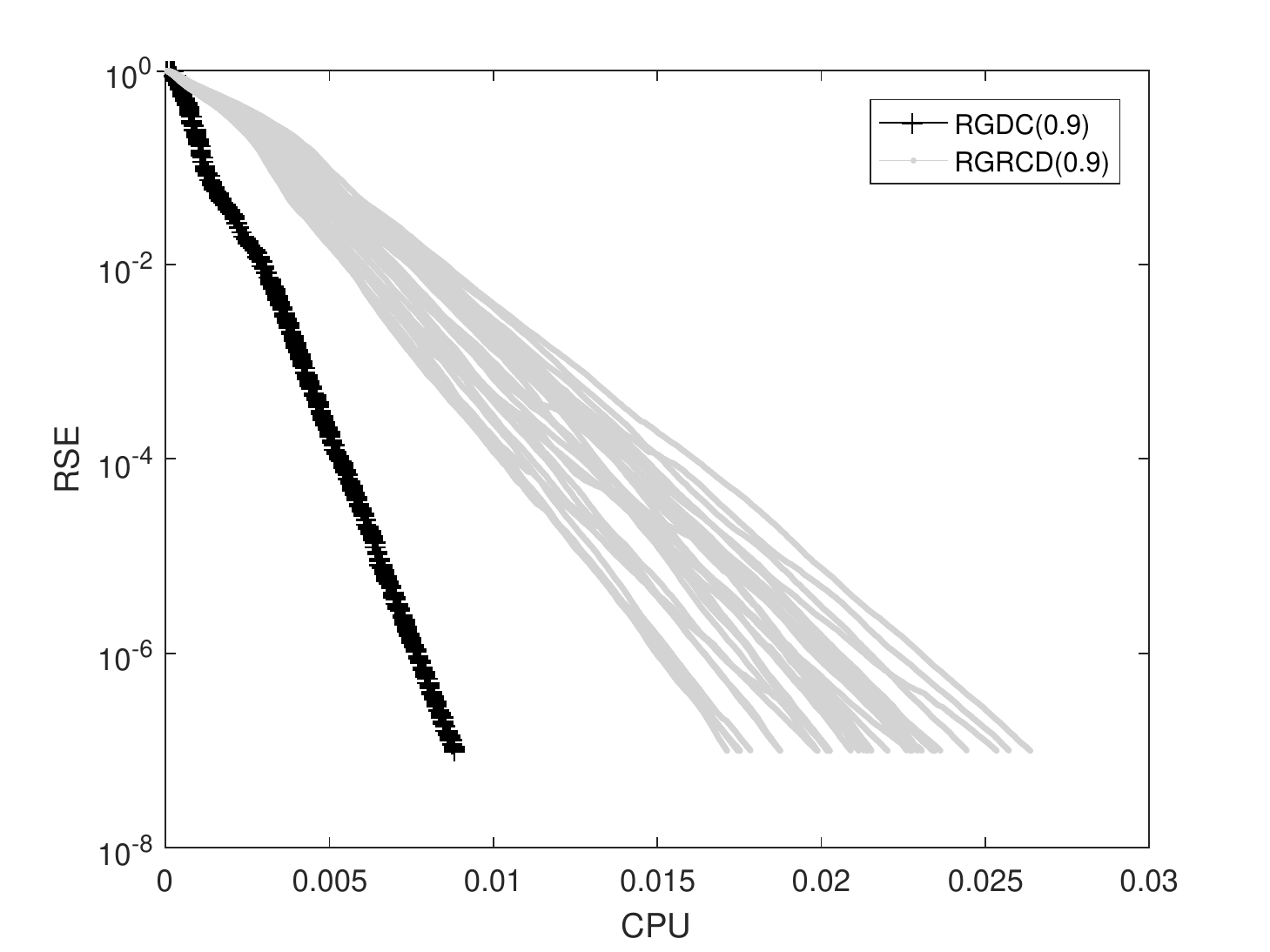}}
\caption{ The convergence curves of RSE  versus  CPU for RGDC($\theta_2$) and RGRCD($\theta_2$) with $\theta_2=0.3$, $0.5$, $0.7$ and $0.9$, where the coefficient matrix is  from Example \ref{Ex:Smatrix} with $A = {\sf Smatrix} (15000, 300, 300, 1.25, 1)$ and $\bb  = A \bx^{\ast} + \bde\bb$.}
  \label{fig:RGDCvsRGRCDnoisy+smatrix}
\end{figure}

\section{Conclusions}
For guaranteeing the fast convergence of row  and column methods, one crucial point is to introduce a practical and appropriate  criterion used to select the working rows or columns from the coefficient matrix.  In this work we present two relaxed greedy deterministic row and column    methods.   Our approaches are based on several ideas and tools, including Petrov-Galerkin conditions and the promising adaptive index selection method, i.e., relaxed greedy selection strategy, which further generalize the RGRK \cite{18BW1}, RGRCD \cite{20ZG} and FDBK \cite{21CH} methods. Our convergence analyses reveal that the resulting algorithms  all have the linear convergence rates, which are bounded by the explicit expressions.

 Theoretically, we can choose the relaxation parameters to be any positive constant from $(0,1)$. However, as  is shown in the previous sections, the convergence rate of RGDR (resp., RGDC) is seriously dependent on the choices of $\theta_1$ (resp., $\theta_2$). In our numerical tests, we take experimentally an exhaustive strategy and find that RGDR (resp., RGDC) works better than RGRK, GBK and RBK (resp., RGRCD, AMDCD and RBCD). It is an important and hard task to find the optimal relaxation parameter which strongly depends on the concrete structures and properties of the coefficient matrix and the relaxed greedy selection strategy, and needs further in-depth study from the viewpoint of
both theory and computations.

\section*{Acknowledgment}
 The authors are very much indebted to the referees for their constructive comments and valuable suggestions, which
greatly improved the original manuscript of this paper.
The work of the first author would like to thank the support of the Fundamental Research Funds for the Central Universities, South-Central University for Nationalities under grant CZQ21027, the second and third authors would like to thank the supports of the National Natural Science Foundation of China under grants  11571265 and 11661161017.


\end{document}